\documentclass[11pt]{article}

\usepackage[a4paper, total={6.7in, 10.7in}]{geometry}
\usepackage[utf8]{inputenc}
\usepackage{xcolor, comment, enumerate, caption, subcaption, tcolorbox, multicol, multirow}
\usepackage{algorithm, algpseudocode}
\usepackage{amssymb, amsmath, amsthm, bbm}
\usepackage{chngcntr}
\usepackage{apptools}
\AtAppendix{\counterwithin{lemma}{section}}
\usepackage[sort,nocompress]{cite}

\usepackage{hyperref}
\hypersetup{
    colorlinks=true,
    allcolors=blue
}

\usepackage[capitalise]{cleveref}
\crefname{equation}{}{}
\crefname{figure}{Figure}{Figures}
\creflabelformat{equation}{\textup{(#2#1#3)}}
\crefname{assumption}{Assumption}{Assumptions}
\crefname{condition}{Condition}{Conditions}
\crefname{property}{Property}{Properties}

\usepackage{enumitem}

\usepackage{pifont}
\newcommand{\cmark}{\ding{51}}%
\newcommand{\xmark}{\ding{55}}%
\crefname{equation}{}{}
\crefname{figure}{Figure}{Figures}
\creflabelformat{equation}{\textup{(#2#1#3)}}
\crefname{assumption}{Assumption}{Assumptions}
\crefname{condition}{Condition}{Conditions}
\crefname{property}{Property}{Properties}

\newtheorem*{theorem*}{Theorem}
\newtheorem{theorem}{Theorem}
\newtheorem{definition}{Definition}

\newtheorem{lemma}{Lemma}
\newtheorem*{lemma*}{Lemma}
\newtheorem{assumption}{Assumption}
\newtheorem{corollary}{Corollary}
\newtheorem*{corollary*}{Corollary}
\newtheorem{condition}{Condition}
\newtheorem{remark}{Remark}
\newtheorem{property}{Property}

\newcommand{\op}[2]{\overset{#1}{#2}}

\newcommand{\He}{\nabla^2 f}
\newcommand{\F}{\mathcal{F}_{0}}
\newcommand{\T}{\mathrm{T}}
\newcommand{\TM}{\mathrm{T}_{\max}}
\newcommand{\f}[1]{\mathbf{#1}}

\newcommand{\lr}[1]{\left\langle #1 \right\rangle}
\newcommand{\krylov}[3]{\mathcal{K}_{#3}(\f{#1}, \f{#2})}

\newcommand{\Rdd}{\mathbb{R}^{d \times d}}
\newcommand{\Rd}{\mathbb{R}^{d}}

\newcommand{\zero}{\f{0}}
\newcommand{\argmin}[1]{\underset{#1}{\text{argmin}}}

\newcommand{\LH}{L_{\mathbf{H}}}

\usepackage{xspace}

\title{Faithful-Newton Framework: Bridging Inner and Outer Solvers for Enhanced Optimization}
\author{Alexander Lim\footnote{School of Mathematics and Physics, University of Queensland, Australia. Email : alexander.lim@uq.edu.au} \and Fred Roosta\footnote{School of Mathematics and Physics, University of Queensland, Australia. ARC Training Centre for Information Resilience (CIRES), Brisbane, Australia. Email : fred.roosta@uq.edu.au}}
\date{\today}

\begin{document}
	\maketitle
\begin{abstract}
        Newton-type methods enjoy fast local convergence and strong empirical performance, but achieving global guarantees comparable to first-order methods remains challenging. Even for simple strongly convex problems, no straightforward variant of Newton's method matches the global complexity of gradient descent. While more sophisticated variants can improve iteration complexity, they typically require solving difficult subproblems with high per-iteration costs, leading to worse overall complexity. These limitations stem from treating the subproblem as an afterthought--either as a black box, yielding overly complex and impractical formulations, or in isolation, without regard to its role in advancing the optimization of the main objective.

        By tightening the integration between the inner iterations of the subproblem solver and the outer iterations of the optimization algorithm, we introduce simple Newton-type variants, called \textit{Faithful-Newton framework}, which in a sense remain faithful to the overall simplicity of classical Newton’s method by retaining simple linear system subproblems. The key conceptual difference, however, is that the quality of the subproblem solution is directly assessed based on its effectiveness in reducing optimality, which in turn enables desirable convergence complexities across a variety of settings.
        Under standard assumptions, we show that our variants, depending on parameter choices, achieve global superlinear convergence, condition-number-independent linear convergence, and/or local quadratic convergence, even when using inexact Newton steps, for strongly convex problems; and competitive iteration complexity for general convex problems. Numerical experiments further demonstrate that our proposed methods perform competitively compared with several alternative Newton-type approaches.
\end{abstract}

\section{Introduction}
Consider the following unconstrained optimization problem
\begin{align}\label{eq:min_f}
    \min_{\f{x}\in\mathbb{R}^d} f(\f{x}),
\end{align}
where the function $f:\Rd \to \mathbb{R}$ is convex and twice continuously differentiable. 
At the risk of oversimplification, optimization algorithms for solving \cref{eq:min_f} can be broadly categorized into first-order methods, 
which are derived from simple gradient descent (GD) \cite{beck2017first,tian2023recent,ruder2016overview,lin2020accelerated,lan2020first,wright2022optimization,bubeck2014theory}, 
and second-order methods, which originate from classical Newton's method \cite{nocedal2006numerical,nesterov2018lectures,bottou2018optimization,boyd2004convex,bertsekas1997nonlinear,sun2006optimization,li2001modified,li2009truncated,fasano2006truncated}. 
For strongly convex problems, classical Newton’s method constructs each step by minimizing a quadratic model of the objective around the current iterate $\f{x}_k$:
\begin{align}\label{eq:quadratic}
\f{d}_{k} = \argmin{\f{s}\in\Rd} , \tfrac{1}{2}\lr{\f{s},\f{H}_{k}\f{s}} + \lr{\f{g}_{k},\f{s}},
\end{align}
followed by $\f{x}_{k+1} = \f{x}_k + \f{d}_k$, where $\f{H}_k$ and $\f{g}_k$ are the Hessian and gradient at $\f{x}_k$. An appealing aspect of this method is that \cref{eq:quadratic} reduces to solving the linear system $\f{H}_k\f{s} = -\f{g}_k$, often called Newton’s system, for which decades of numerical linear algebra research offer highly effective solution techniques \cite{saaditerative2003,bjorck2015numerical}.

When the initial point $\f{x}_0$ is close to the optimum, classical Newton's method achieves quadratic convergence \cite{nesterov2018lectures,nocedal2006numerical}. However, if initialized far away, it may diverge. To address this, damped Newton's method \cite{nocedal2006numerical,boyd2004convex,nesterov2018lectures} introduces a step-size $\eta$ satisfying a line-search condition, such as Armijo’s rule,  
\begin{align}\label{eq:armijo}
    f(\f{x}_k + \eta\f{d}_k) \leq f(\f{x}_k) + \rho\lr{\f{g}_k,\eta\f{d}_k}, \quad \rho > 0,
\end{align}
which ensures sufficient descent and hence global convergence.  It is well-known that, locally, the unit step-size is always accepted by \cref{eq:armijo}, leading to the Newton direction being naturally scaled, with $\eta = 1$ regarded as the natural step size of Newton's method---a property shared predominantly by second-order methods \cite{smee2025first}.

Damped Newton's method typically outperforms first-order methods in convex settings \cite{xu2020second,xu2016sub,yao2021adahessian}, and its strong empirical performance persists even when steps are approximated, i.e., $\f{d}_k \approx -\f{H}_k^{-1}\f{g}_k$. This observation motivates truncated, or inexact, Newton's methods \cite{dembo1982inexact,dembo1983truncated,dixon1988numerical,nash1991numerical}, where an iterative solver approximates the Newton step, and the residual condition $\|\f{g}_k + \f{H}_k\f{s}\| \leq \omega \|\f{g}_k\|$ for some $0 < \omega < 1$, is used to measure inexactness. Thus, the Newton system can be solved only approximately without significant loss of performance, enabling Newton-type methods to scale to high-dimensional problems \cite{lim2025complexity,xu2020second,yao2023inexact,roosta2022newton,royer2020newton,roosta2019sub}.  

From the brief introduction above, the appealing features of inexact/damped Newton's method include theoretical intuitiveness, simple linear subproblems, scalability to high-dimensional problems, natural scaling of the update direction, and strong empirical performance. Nonetheless, establishing global convergence results that match or outperform simpler first-order alternatives remains challenging. As noted by many \cite{kamzolov2024optami,doikov2024gradient,nesterov2018lectures,boyd2004convex}, the theoretical global convergence rates of the damped Newton's method are typically worse than those of GD. Moreover, to our knowledge, even for simple strongly convex problems, no existing analysis of classical Newton's method matches the global iteration/operation complexity of GD.

\begin{table}[htbp]
    \centering
    \renewcommand{\arraystretch}{1.2}
    \resizebox{\columnwidth}{!}{
    \begin{tabular}{|c|c|c|c|c|}
        \hline
        \multirow{2}{*}{\textbf{Methods}} & \multirow{2}{*}{\shortstack{\textbf{$p$\textsuperscript{th}-Lipschitz} \\ \textbf{Smoothness}}} & \multirow{2}{*}{\shortstack{\textbf{Linear}\\\textbf{Subproblems*}}} & \multirow{2}{*}{\shortstack{\textbf{Inexact} \\\textbf{Subproblems}}} & \multirow{2}{*}{\shortstack{\textbf{Iteration}\\\textbf{Complexity}}} \\
        & & & & \\
        \hline
        \multicolumn{5}{|c|}{\textbf{Strongly Convex}} \\
        \hline
        \multirow{2}{*}{Lower Bound \cite{arjevani2019oracle}} & \multirow{2}{*}{2} & \multicolumn{2}{|c|}{} & $\Omega((\LH R/\mu)^{2/7} $\\
         & & \multicolumn{2}{|c|}{} & $+ \ln\ln(\mu^3(\LH^2\varepsilon)^{-1}))$ \\
        \hline
        Gradient descent & 1 & - & - & $\mathcal{O}(\kappa\ln(\varepsilon^{-1}))$ \\
        NAG \cite{nesterov2018lectures} & 1 & - & - & $\mathcal{O}(\sqrt{\kappa}\ln(\varepsilon^{-1}))$ \\
        Inexact Newton's method \cite{roosta2019sub} & 1 & \color{green}{\cmark} & \color{green}{\cmark} & $\mathcal{O}(\kappa^2\ln(\varepsilon^{-1}))$ \\
        CRN \cite{kamzolov2024optami} & 2 & \color{red}{\xmark} & \color{red}{\xmark} & Superlinear**\\
        UN \cite{hanzely2024newton} & 2 & \color{red}{\xmark} & \color{red}{\xmark} & Superlinear**\\
        \hline
         & \multirow{6}{*}{2} & \multirow{6}{*}{\color{green}{\cmark}} & \multirow{6}{*}{\color{green}{\cmark}} & {Superlinear$^{\dagger}$} \\[-5pt]
         & & & & {(\cref{cor:superlinear})}\\
         \textbf{FNCR-LS} & & & & $\mathcal{O}(\ln(\varepsilon^{-1}))^{\dagger}$ \\[-5pt]
         (\cref{alg:fncr-ls}) & & & & (\cref{cor:superlinear_linear_fncr-ls}) \\
         & & & & $\mathcal{O}(\ln(\ln(\varepsilon^{-1})))^{\dagger\dagger}$ \\[-5pt]
         & & & & (\cref{thm:fncr:local:quadratic}) \\
        \hline
        \multicolumn{5}{|c|}{\textbf{Convex}} \\
        \hline
        Lower Bound \cite{arjevani2019oracle} & 2 & \multicolumn{2}{|c|}{} & $\Omega(\varepsilon^{-2/7})$\\
        \hline
        Inexact Newton's method & \multicolumn{4}{|c|}{Convergence cannot be guaranteed} \\
        \hline
        \textbf{FNCR-reg-LS} & \multirow{2}{*}{2} & \multirow{2}{*}{\color{green}{\cmark}} & \multirow{2}{*}{\color{green}{\cmark}} & {$\mathcal{O}(\varepsilon^{-1/2})$}\\[-5pt]
        (\cref{alg:fncr-ls}) & & & & {(\cref{thm:reg_fncr})}\\
        \hline
        Gradient Reg \cite{doikov2024gradient,mishchenko2023regularized,doikov2024super} & 2 & \color{red}{\xmark} & \color{red}{\xmark} & \multirow{2}{*}{$\mathcal{O}(\varepsilon^{-1/2})$} \\
        CRN \cite{nesterov2006cubic} & 2 & \color{red}{\xmark} & \color{red}{\xmark} & \\
        \hline
        Nesterov's Acc \cite{nesterov2008accelerating} & 2 & \color{red}{\xmark} & \color{red}{\xmark} & \multirow{2}{*}{$\mathcal{O}(\varepsilon^{-1/3})$} \\
        Contracting Proximal \cite{doikov2024gradient,doikov2020contracting} & 2 & \color{red}{\xmark} & \color{green}{\cmark} & \\
        \hline
        Acc Taylor Descent \cite{bubeck2019near} & 2 & \color{red}{\xmark} & \color{red}{\xmark} & \multirow{2}{*}{$\tilde{\mathcal{O}}(\varepsilon^{-2/7})$} \\
        NATMI \cite{kamzolov2020near} & 2 & \color{red}{\xmark} & \color{green}{\cmark} & \\
        \hline
        Optimal Tensor method \cite{kovalev2022first} & 2 & \color{red}{\xmark} & \color{red}{\xmark} & \multirow{3}{*}{$\mathcal{O}(\varepsilon^{-2/7})$} \\
        A-NPE \cite{monteiro2013accelerated} & 2 & \color{red}{\xmark} & \color{green}{\cmark} & \\
        Opt Acc \cite{carmon2022optimal} & 1,2 & \color{red}{\xmark} & \color{green}{\cmark} & \\
        \hline
    \end{tabular}
    }
    \caption{Overview of classical and modern higher-order optimization methods for solving \cref{eq:min_f}, compared against classical first-order methods. Here, $\kappa$ is the condition number, $\varepsilon$ the global suboptimality level (see \cref{def:suboptimality}), and $R \triangleq \|\f{x}_0 - \f{x}^*\|$. The notation $\tilde{\mathcal{O}}$ hides logarithmic factors. (*) Only one linear subproblem is solved per iteration. (**) Superlinear convergence holds for strongly star-convex functions. ($\dagger$) Global convergence. ($\dagger\dagger$) Local convergence for $\|\mathbf{g}_k\| \leq r$ with some $r > 0$ from \cref{eq:r}: depending on parameter choices, either a condition-number–independent linear rate or a quadratic rate is obtained.}
    \label{table:comparison}
\end{table}

To address this conundrum, many sophisticated variants of Newton's method have been developed to achieve state-of-the-art iteration complexity, sometimes matching or even surpassing that of GD and its variants. These include (adaptive) cubic regularized Newton's methods \cite{nesterov2006cubic,cartis2011adaptive,kamzolov2024optami} and gradient-regularized Newton's methods \cite{mishchenko2023regularized,doikov2024super,doikov2024gradient}, as well as accelerated variants such as accelerated cubic regularized Newton's methods \cite{nesterov2008accelerating}, accelerated Taylor descent \cite{bubeck2019near}, and accelerated Newton proximal extragradient methods \cite{monteiro2013accelerated,carmon2022optimal}, along with other notable efforts \cite{gasnikov2019near,hanzely2022damped,hanzely2024newton,kovalev2022first}.  

Unfortunately, these improved convergence rates often come at the cost of significantly increased computational complexity, as the methods involve more intricate iterations with higher per-iteration costs. Many appealing features of inexact Newton's method—such as simplicity, scalability, and empirical effectiveness—are no longer preserved. In essence, the gains in iteration complexity come at the expense of simple linear subproblems and, at times, even inexactness, making these methods challenging to implement, less effective empirically, impractical for high-dimensional problems, and in some cases worse than simple GD in terms of operation complexity\footnote{Operation complexity refers to the total number of function, gradient, and Hessian-vector product evaluations required to achieve a desired suboptimality.}. \cref{table:comparison} contrasts some of these recent developments with classical inexact Newton’s method and several standard first-order methods. In this light, this paper seeks to address the following question:

\begin{quote}
    \textit{Is there a close variant of the inexact Newton’s method with a simple linear subproblem that improves upon---or at least matches---the theoretical convergence of gradient descent for (strongly) convex problems, while still offering superior empirical performance?}
\end{quote}
To that end, we develop the \emph{Faithful-Newton} (FN) framework\footnote{The term ``faithful'' reflects the framework's adherence to the essence of classical Newton's method, e.g., solving one linear subproblem per iteration, maintaining a natural unit step-size, and using the local quadratic model.}. FN is not a method per se, but a framework that allows for different instantiations and is motivated as follows. In the literature, Newton-type methods are often studied as two-level schemes: the outer solver corresponds to the Newton-type method itself, where analyses (e.g., worst-case complexity) assume an inexact Newton step $\f{d}_k$ satisfying a specific inexactness condition has been obtained; the inner solver computes such a step, but is typically studied independently and often overlooked, with little impact on the outer solver's analysis. FN departs from this separation by tightly integrating inner and outer solvers, so that the choice of inner solver directly affects the convergence guarantees of the outer solver. In this work, we instantiate FN with the conjugate residual (CR) method \cite{lim2024conjugate} as the inner solver, leveraging its specific properties to obtain enhanced convergence guarantees.\bigskip

\noindent
\textbf{Contributions.} Our contributions can be summarized as follows: 
\begin{enumerate}
    \item Using conjugate residual (CR) as inner solver, we present an instantiation of the FN framework, called Faithful-Newton-Conjugate-Residual with Line-search (FNCR-LS). In strongly convex settings with Lipschitz continuous Hessians, FNCR-LS achieves, depending on parameter choices, condition-number-independent linear convergence (\cref{cor:superlinear_linear_fncr-ls}), global superlinear convergence (\cref{cor:superlinear}), and/or local quadratic convergence (\cref{thm:fncr:local:quadratic}).

    \item We then extend the analysis of FNCR-LS to general convex settings, which requires regularizing the Hessian. We denote this variant as FNCR-reg-LS and show that it achieves an iteration complexity of $\mathcal{O}(1/\sqrt{\varepsilon})$, matching the best known rate among non-accelerated second-order methods (\cref{thm:reg_fncr}).

    \item Finally, in \cref{sec:num_exp}, we present experiments comparing FN with other second-order algorithms. The results show that both FNCR-LS and FNCR-reg-LS are highly competitive with existing second-order methods.
\end{enumerate}

\vspace{2mm}
We organize the rest of the paper as follows. We end this section by introducing the notations and definitions. Before presenting our method, we provide a brief introduction to our main choice of inner solver, CR, in \cref{sec:cr}. In \cref{sec:fncr_ls}, we introduce FNCR-LS as an instantiation of the FN framework, and analyze it under strongly convex problems (\cref{sec:strongly_convex_functions}). We then extend the analysis to general convex functions using FNCR-reg-LS (\cref{sec:general_convex_functions}). Finally, in \cref{sec:num_exp}, we compare the empirical performance of FNCR-LS and FNCR-reg-LS with other second-order methods.\\

\noindent
\textbf{Notations and Definitions.} 
Throughout the paper, vectors and matrices are denoted by bold lower- and upper-case letters, respectively, e.g., $\f{g}$ and $\f{H}$. Their norms, e.g., $\|\f{g}\|$ and $\|\f{H}\|$, correspond to the vector Euclidean and matrix spectral norms. Scalars are represented by regular letters, e.g., $n$, $L$, $\sigma$, $\alpha$, $\beta$, etc. We denote the optimal function value as $f^* \triangleq f(\f{x}^*) \triangleq \min f(\f{x})$, and the sub-level set for a given initial point $\f{x}_0$ as $\F \triangleq \{\f{x} \, | \, f(\f{x}) \leq f(\f{x}_0)\}$.  

The gradient at $\f{x}$ is denoted by $\f{g}(\f{x}) \triangleq \nabla f(\f{x})$, or simply $\f{g}$. The (regularized) Hessian at $\f{x}$ is denoted by $\f{H}(\f{x})$, or simply $\f{H}$, i.e., $\nabla^2 f(\f{x}) + \sigma \f{I}$ or $\nabla^2 f(\f{x})$, depending on the context specified in the respective sections. The maximum and minimum eigenvalues of a matrix $\f{H}$ are denoted by $\lambda_{\max}$ and $\lambda_{\min}$, respectively, and the condition number is $\kappa \triangleq \lambda_{\max}/\lambda_{\min}$. We write $\f{H} \succ \f{0}$ to indicate that $\f{H}$ is positive definite.  

We often discuss the interaction between the outer solver (i.e., optimization method) and the inner solver (i.e., the CR method). The iteration counters for these solvers are denoted by $k$ and $t$, respectively. For example, $\f{x}_k$, $\f{g}_k$, and $\f{H}_k$ correspond to the outer solver, while $\f{d}_k = \f{s}_k^{(t)}$ indicates that the direction $\f{s}_k^{(t)}$, generated by CR at its $t$-th iteration, is used as the search direction for the $k$-th outer iteration. In most cases, the context allows dropping the outer iteration counter and writing simply $\f{s}_k$, with the superscript reintroduced when necessary. We also denote $\delta(\f{x}) \triangleq f(\f{x}) - f^*$ and, with the outer iteration counter, $\delta_k \triangleq \delta(\f{x}_k)$.  

The Krylov subspace of degree $t \geq 1$ generated by $\f{H}$ and $\f{g}$ is denoted by 
\[
\krylov{H}{g}{t} = \text{Span}\{\f{g},\f{Hg}, \f{H}^2\f{g}, \ldots, \f{H}^{t-1}\f{g}\}.
\] 
The grade of $\f{g}$ with respect to $\f{H}$ is a positive integer $g$ such that
\begin{align*}
    \text{dim}(\krylov{H}{g}{t}) = 
    \begin{cases}
        t, & t \leq g,\\
        g, & t > g.
    \end{cases}
\end{align*}

\begin{definition}[$\mu$-Convexity]
    A function $f$ is $\mu$-convex if there exists $0 \leq \mu < \infty$ such that, for any $\f{x}, \f{y} \in \mathbb{R}^{d}$,
    \begin{align}\label{eq:convexity}
        f(\f{y}) \geq f(\f{x}) + \lr{\f{g}(\f{x}), \f{y} - \f{x}} + \frac{\mu}{2}\|\f{y} - \f{x}\|^2.
    \end{align}
    Furthermore, $f$ is called strongly convex if $\mu > 0$, and general convex if $\mu = 0$.
\end{definition}

\begin{definition}[$\varepsilon$-Suboptimality]
\label{def:suboptimality}
    Given $0 < \varepsilon < 1$, a point $\f{x} \in \mathbb{R}^{d}$ is an $\varepsilon$-suboptimal solution to \cref{eq:min_f} if $\delta(\f{x}) \leq \varepsilon$.
\end{definition}

\begin{condition}[$c$-Sufficient Reduction Condition]\label{cond:sufficient_reduction_condition}
    For any $\f{x} \in \mathbb{R}^{d}$, a vector $\f{d} \in \mathbb{R}^{d}$ is said to be $c$-sufficient with respect to a surrogate model $m_{\f{x}}$, with $c > 0$, if
    \begin{align}\label{eq:faithful_cond}
       c \leq \frac{f(\f{x}) - f(\f{x} + \f{d})}{m_{\f{x}}(\f{0}) - m_{\f{x}}(\f{d})}.
    \end{align}
\end{condition}
In this paper, we focus exclusively on the linear surrogate model
\[
m_{\f{x}}(\f{d}) = f(\f{x}) + \lr{\f{g}(\f{x}), \f{d}}.
\]
When the surrogate model $m_{\f{x}}$ is linear, \cref{cond:sufficient_reduction_condition} reduces to the Armijo condition \cref{eq:armijo} with $c = \rho$. While \cref{cond:sufficient_reduction_condition} is standard in optimization, we name it here to facilitate discussions about the directions that are $c_1$-sufficient but not $c_2$-sufficient. 

As noted by \cite{arjevani2019oracle}, for second-order information to be beneficial, some regularity condition on the Hessian is essential. A common choice is the Hessian-Lipschitz smoothness assumption. In fact, the gradient-Lipschitz condition has recently come under increasing scrutiny in machine learning, as many common objectives fail to satisfy it \cite{patelGlobalConvergenceStability2022}, and it may not hold even along the optimization trajectory \cite{cohenGradientDescentNeural2022}. As an alternative, \cite{ahnUnderstandingUnstableConvergence2022} advocates for Lipschitz continuity of the Hessian, arguing that it provides a more realistic regularity condition in such settings. Therefore, in this paper, we focus on functions satisfying Hessian-Lipschitz smoothness.

\begin{assumption}[$\LH$-Lipschitz Hessian] \label{assmp:LH}
    The function $f$ is twice continuously differentiable and bounded below. Moreover, there exists $0 \leq \LH < \infty$ such that
    \[
        \|\He(\f{x}) - \He(\f{y})\| \leq \LH \|\f{x} - \f{y}\|, \quad \forall \f{x}, \f{y} \in \mathbb{R}^{d}.
    \]
\end{assumption}
\cref{assmp:LH} implies that for all $\f{d} \in \mathbb{R}^{d}$,
\begin{subequations}\label{eq:LH}
\begin{align}
    &f(\f{x} + \f{d}) \leq f(\f{x}) + \lr{\f{g}(\f{x}), \f{d}}  + \frac{1}{2}\lr{\f{d},\He(\f{x})\f{d}} + \frac{\LH}{6}\|\f{d}\|^3\label{eq:assmp:LH:fkp1<fk};\\
    &\left\|\f{g}(\f{x} + \f{d}) - \f{g}(\f{x}) - \He(\f{x})\f{d}\right\|  \leq \frac{\LH}{2}\|\f{d}\|^2.\label{eq:assmp:LH:gkp1<p}
\end{align}
\end{subequations}
Proofs of \cref{eq:assmp:LH:fkp1<fk,eq:assmp:LH:gkp1<p} can be found in \cite{nesterov2018lectures}. 

\section{The Conjugate Residual Method}\label{sec:cr}
We now turn our focus to solving the Newton's system. Since we are concentrating on solving the linear system\footnote{In this section, we treat $\f{H}$ and $\f{g}$ as a generic positive definite matrix in $\Rdd$ and a vector in $\Rd$, respectively, which appear in a linear system we aim to solve. They do not necessarily have any geometric meaning or bear any direct relation to our optimization algorithm.}, we temporarily drop the subscript for readability.

Conjugate Residual (\cref{alg:cr}) is an iterative method developed by Hestenes and Stiefel \cite{hestenes1952methods}, originally designed to solve a linear system with $\f{H} \succ \f{0}$. It has since been studied further in the context of semidefinite systems \cite{hayami2001behaviour,hayami2011geometric,hayami2004convergence} and even (singular) indefinite systems \cite{lim2024conjugate,luenberger1970conjugate,greenbaum1997iterative}.

\begin{algorithm}[htbp]
	\caption{Conjugate Residual}\label{alg:cr}
	\begin{algorithmic}[1]
		\Require $\f{H}$, $\f{g}$
		\State $t = 0$, $\f{s}^{(0)} = \zero$, $\f{r}^{(0)} = \f{p}^{(0)} = -\f{g}$
		\While{\textbf{Not Terminated}}
		\State $\alpha^{(t)} = \lr{\f{r}^{(t)},\f{Hr}^{(t)}}/\|\f{Hp}^{(t)}\|^2$
		\State $\f{s}^{(t+1)} = \f{s}^{(t)} + \alpha^{(t)}\f{p}^{(t)}$
        \State $\f{r}^{(t+1)} = \f{r}^{(t)} - \alpha^{(t)}\f{Hp}^{(t)}$
		\State $\gamma^{(t)} = \lr{\f{r}^{(t+1)},\f{Hr}^{(t+1)}}/\lr{\f{r}^{(t)},\f{Hr}^{(t)}}$
		\State $\f{p}^{(t+1)} = \f{r}^{(t)} + \gamma^{(t)}\f{p}^{(t)}$
		\State $t = t + 1$
		\EndWhile
		\State \Return $\f{s}^{(t)}$
	\end{algorithmic}
\end{algorithm}
At each iteration $t$, CR finds an approximation $\f{s}^{(t)}$ to the solution of the linear system within the Krylov subspace $\krylov{H}{g}{t}$, with the constrain that the residual vector $\f{r}^{(t)} \triangleq -\f{g} - \f{Hs}^{(t)}$ is orthogonal to the subspace $\f{H}\krylov{H}{g}{t}$. To put it formally, the two conditions can be written as, 
\begin{align*}
    \f{s}^{(t)} \in \krylov{H}{g}{t},\quad \text{and} \quad 
    \f{r}^{(t)} \perp \f{H}\krylov{H}{g}{t}.
\end{align*}
The latter condition is also referred to as Petrov-Galarkin condition. Imposing the above two conditions results in a series of desirable CR properties.
\begin{property}\label{prop:cr_prop}
    Let $\f{H} \succ \zero \in \Rdd$ and $\f{g} \neq \f{0}$. For any $1 \leq t \leq g$ in \cref{alg:cr}, we have 
        \begin{subequations}
        \begin{align}
            \|\f{r}^{(t)}\| & < \|\f{r}^{(t-1)}\|;\label{eq:rt>rtp1}\\
            \|\f{s}^{(t-1)}\| & < \|\f{s}^{(t)}\|;  \label{eq:stp1>st}\\
            \|\f{Hs}^{(t-1)}\| & < \|\f{Hs}^{(t)}\| \leq \|\f{g}\|; \label{eq:Hstp1>Hst}\\
            \f{s}^{(g)} & = - \f{H}^{-1}\f{g} \label{eq:sg=-Hg}\\
            \lr{\f{g},\f{s}^{(t)}} & < \lr{\f{g},\f{s}^{(t-1)}} < 0;\label{eq:stp1g>stg}\\
            \|\f{Hp}^{(t)}\| & \leq \|\f{Hr}^{(t)}\|; \label{eq:Hr>Hp}\\
            \lr{\f{g}, \f{s}^{(t)}} & \leq -\lr{\f{s}^{(t)},\f{Hs}^{(t)}} < 0,\label{eq:gst<-sHs}\\
            0 & = \lr{\f{s}^{(t)}, \f{Hr}^{(t)}} \label{eq:sHr=0}
        \end{align}
        \end{subequations}
    \noindent where the inequalities in \cref{eq:Hr>Hp,eq:gst<-sHs} are strict for $1 \leq t \leq g-1$.
\end{property}
The properties shown in \cref{prop:cr_prop} are well known within the numerical linear algebra community. Their proofs can be found in many standard texts and references \cite{dahito2019conjugate,saaditerative2003,liu2022minres,lim2025complexity}, and are, therefore, omitted. In contrast, the following properties of CR are less well known, and we include their proofs here for convenience.

\begin{lemma}\label{lemma:alpha_ineq}
    Let $\f{H} \succ \f{0}$ and $\f{g}$ be any vector. In \cref{alg:cr}, we have 
    \begin{align}\label{eq:lemma:alpha_ineq}
        \lambda_{\max}^{-1} & \leq \alpha^{(t)} \leq \lambda_{\min}^{-1},
    \end{align}    
    where $\alpha^{(t)}$ is the scalar generated by \cref{alg:cr} and $0 \leq t \leq g - 1$.
\end{lemma}
\begin{proof}
     For the lower bound, we have
    \begin{align*}
        \frac{1}{\lambda_{\max}} \leq \frac{\lr{\f{Hr}^{(t)}, \f{H}^{-1}\f{Hr}^{(t)}}}{\|\f{Hr}^{(t)}\|^2} = \frac{\lr{\f{r}^{(t)}, \f{Hr}^{(t)}}}{\|\f{Hr}^{(t)}\|^2} \op{\cref{eq:Hr>Hp}}{\leq} \frac{\lr{\f{r}^{(t)}, \f{Hr}^{(t)}}}{\|\f{Hp}^{(t)}\|^2} = \alpha^{(t)},
    \end{align*}
    where the first inequality follows from $\lr{\f{Hr}^{(t)}, \f{H}^{-1}\f{Hr}^{(t)}} \geq \lambda_{\max}^{-1}\|\f{Hr}^{(t)}\|^2$. For the upper bound, we have 
    \begin{align*}
        \frac{\lr{\f{r}^{(t)}, \f{Hr}^{(t)}}^2}{\|\f{Hp}^{(t)}\|^2} & = \frac{\lr{\f{r}^{(t)}, \f{H}(\f{r}^{(t)} + \beta^{(t-1)}\f{p}^{(t-1)})}^2}{\|\f{Hp}^{(t)}\|^2} = \frac{\lr{\f{r}^{(t)}, \f{Hp}^{(t)}}^2}{\|\f{Hp}^{(t)}\|^2} \leq \|\f{r}^{(t)}\|^2,
    \end{align*}
    which in turn implies
    \begin{align*}
        \alpha^{(t)} & = \frac{\lr{\f{r}^{(t)}, \f{Hr}^{(t)}}}{\|\f{Hp}^{(t)}\|^2} \leq \frac{\|\f{r}^{(t)}\|^2}{\lr{\f{r}^{(t)}, \f{Hr}^{(t)}}} \leq \frac{1}{\lambda_{\min}}.
    \end{align*}
\end{proof}
The following property is a stronger version of \cref{eq:stp1g>stg}, which will be utilized in later sections.
\begin{lemma}\label{lem:bst>bstm1+lnr}
    Let $\f{H} \succ \f{0}$ and $\f{g}$ be any vector. For $1 \leq t \leq g$, in \cref{alg:cr}, we have
    \begin{align}\label{eq:lem:gst<gstm1-lr_weaker}
        \lr{\f{g},\f{s}^{(t)}} & \leq \lr{\f{g},\f{s}^{(t-1)}} - \lambda_{\max}^{-1}\|\f{r}^{(t-1)}\|^2.
    \end{align}
\end{lemma}
\begin{proof}
    By the construction of $\f{s}^{(t)} = \f{s}^{(t-1)} + \alpha^{(t-1)}\f{p}^{(t)}$, we note
    \begin{align*}
        \lr{\f{g},\f{s}^{(t)}} & = \lr{\f{g},\f{s}^{(t-1)}} + \alpha^{(t-1)}\lr{\f{g},\f{p}^{(t-1)}}\\
        & = \lr{\f{g},\f{s}^{(t-1)}} - \alpha^{(t-1)}\lr{\f{r}^{(t-1)},\f{Hr}^{(t-1)}}\sum_{i=0}^{t-1}\frac{\|\f{r}^{(i)}\|^2}{\lr{\f{r}^{(i)},\f{Hr}^{(i)}}},
    \end{align*}
    where the second equality is obtained from expanding $\f{p}^{(t-1)}$. If we drop all terms except for the $(t-1)$\textsuperscript{th} term from the summand, we obtain
    \begin{align*}
        \lr{\f{g},\f{s}^{(t)}} & = \lr{\f{g},\f{s}^{(t-1)}} - \alpha^{(t-1)}\lr{\f{r}^{(t-1)},\f{Hr}^{(t-1)}}\sum_{i=0}^{t-1}\frac{\|\f{r}^{(i)}\|^2}{\lr{\f{r}^{(i)},\f{Hr}^{(i)}}}\\
        & \leq \lr{\f{g},\f{s}^{(t-1)}} - \alpha^{(t-1)}\|\f{r}^{(t-1)}\|^2 \op{\cref{eq:lemma:alpha_ineq}}{\leq} \lr{\f{g},\f{s}^{(t-1)}} - \lambda_{\max}^{-1}\|\f{r}^{(t-1)}\|^2.
    \end{align*}
\end{proof}

\section{Faithful-Newton-Conjugate-Residual with Line-search}\label{sec:fncr_ls}
In this section, we introduce Faithful-Newton-Conjugate-Residual with line-search (FNCR-LS), \cref{alg:fncr-ls}, as one particular instantiation of the FN framework.
\begin{algorithm}[htbp]
\caption{Faithful-Newton-Conjugate-Residual with Line-Search (FNCR-LS)}
\label{alg:fncr-ls}
\begin{algorithmic}[1]
    \Require $\f{x}_0$ \hfill \Comment{Initial iterate}
    \State $k \gets 0$
    \While{termination criterion not met}
        \State $\f{s}_k \gets$ \cref{alg:cr_with_sd_si} with $(\f{H}_k, \f{g}_k)$ \hfill \Comment{CR solver}
        \State $\f{d}_k \gets$ \cref{alg:backtracking} with $\f{s}_k$ \hfill \Comment{Line-search}
        \State $\f{x}_{k+1} \gets \f{x}_k + \f{d}_k$
        \State $k \gets k+1$
    \EndWhile
    \State \Return $\f{x}_k$ \hfill \Comment{Final iterate}
\end{algorithmic}
\end{algorithm}

\begin{algorithm}[ht]
\caption{Conjugate Residual with Sufficient Iteration and Descent Conditions}
\label{alg:cr_with_sd_si}
\begin{algorithmic}[1]
    \Require $\f{H}$, $\f{g}$, $1 \leq \T \leq \TM \leq d$, $0 < \rho < 1/2$, $0 \leq \omega < 1$
    \State $t \gets 0$, $\f{s}^{(0)} \gets \zero$, $\f{r}^{(0)} \gets \f{p}^{(0)} \gets -\f{g}$, $\rho_0 \gets \rho$
    \While{$(t < \T)$ \textbf{or} ($\f{s}^{(t)}$ is $\rho_t$-sufficient)} \hfill \Comment{Sufficient Iteration \& Descent conds.}
        \If{$\|\f{r}^{(t)}\| \leq \omega \|\f{g}\|$ \textbf{or} $t = \TM$} \hfill \Comment{Residual \& Max. iteration conds.}
            \State \Return $\f{s}^{(t)}$, \texttt{dType} = \texttt{TER}
        \EndIf
        \State $\alpha^{(t)} \gets {\lr{\f{r}^{(t)},\f{H}\f{r}^{(t)}}}/{\|\f{H}\f{p}^{(t)}\|^2}$ \hfill \Comment{Standard CR iteration (Lines 6-11)}
        \State $\f{s}^{(t+1)} \gets \f{s}^{(t)} + \alpha^{(t)}\f{p}^{(t)}$
        \State $\f{r}^{(t+1)} \gets \f{r}^{(t)} - \alpha^{(t)}\f{H}\f{p}^{(t)}$
        \State $\gamma^{(t)} \gets {\lr{\f{r}^{(t+1)},\f{H}\f{r}^{(t+1)}}}/{\lr{\f{r}^{(t)},\f{H}\f{r}^{(t)}}}$
        \State $\f{p}^{(t+1)} \gets \f{r}^{(t+1)} + \gamma^{(t)}\f{p}^{(t)}$
        \State $t \gets t+1$
        \State $\rho_t \gets \rho\|\f{g}\|^2 / \|\f{r}^{(t-1)}\|^2$ \Comment{Adaptive sufficiency parameter}
    \EndWhile
    \If{$t = \T$}
        \State \Return $\f{s}^{(t)}$, \texttt{dType} = \texttt{INS}
    \Else
        \State \Return $\f{s}^{(t-1)}$, \texttt{dType} = \texttt{SUF}
    \EndIf
\end{algorithmic}
\end{algorithm}

\begin{algorithm}[htbp]
\caption{Backtracking Line-search}
\label{alg:backtracking}
\begin{algorithmic}[1]
    \Require $\f{s}$, $0 < \rho < 1/2$, $0 < \zeta < 1$, $\eta_0 = 1$
    \State $j \gets 0$, $\eta \gets \eta_0$
    \While{$\eta \f{s}$ is not $\rho$-sufficient}
        \State $\eta \gets \eta_0 \zeta^j$
        \State $j \gets j+1$
    \EndWhile
    \State \Return $\eta \f{s}$
\end{algorithmic}
\end{algorithm}

FNCR-LS integrates CR (inner iterations) and Newton's method (outer iterations) while incorporating a backtracking line-search routine.  
The key difference between FNCR-LS and the classical inexact Newton's method lies in how the inner solver CR is utilized.  
Specifically, FNCR-LS augments the classical CR with two additional mechanisms---sufficient descent and sufficient iteration conditions---to integrate the outer and inner solvers; see \cref{alg:cr_with_sd_si} where these specific mechanisms can be seen in line 2. The first and second conditions, separated by a disjunction, are the sufficient iteration and sufficient descent conditions, respectively.
FNCR-LS can be seen as a generalized inexact Newton's method. With specific hyperparameter settings, FNCR-LS can be reduced to several familiar Newton-type methods, including inexact/damped Newton's method. 

The algorithmic procedure arising from the coupling of \cref{alg:fncr-ls} (outer iterations) and \cref{alg:cr_with_sd_si} (inner iterations) can be described as follows. For ease of exposition of FNCR-LS, suppose $\T = 1$ and $\rho_t = \rho$ in \cref{alg:cr_with_sd_si}, i.e., with the sufficient iteration condition disabled. 
At each outer iteration $k$, FNCR-LS (\cref{alg:fncr-ls}) calls the modified CR solver (\cref{alg:cr_with_sd_si}) to approximately solve the Newton system, starting from $\f{s}_k^{(0)} = \f{0}$.
For each inner iteration $t \geq 1$, the solver checks whether $\f{s}_k^{(t)}$ is $\rho$-sufficient for some $0 < \rho < 1/2$.
If so, it proceeds with a standard CR update. The process stops when:
(1) $\f{s}_k^{(t)}$ is no longer $\rho$-sufficient,
(2) an approximate solution is found, i.e., $\|\f{g}_k + \f{H}_k\f{s}_k^{(t)}\| \leq \omega\|\f{g}_k\|$, or
(3) the maximum number of inner iterations $\TM$ is reached.

In case (1), if $t>1$, the previous iterate $\f{s}_k^{(t-1)}$ (which was $\rho$-sufficient) is returned as a \texttt{SUF}-type vector; if $t=1$, the method instead returns $\f{s}_k^{(1)} = -\alpha_k^{(0)} \f{g}_k$ as \texttt{INS}-type.
In FNCR-LS, a \texttt{SUF}-type vector is accepted directly as $\f{d}_{k}$ (skipping backtracking) and used to update $\f{x}_{k+1} = \f{x}_k + \f{d}_k$, while an \texttt{INS}-type vector triggers the backtracking routine (\cref{alg:backtracking}).
In cases (2) or (3), the solver returns $\f{s}_k^{(t)}$ as \texttt{TER}-type.
For $\T = 1$, since the sufficiency check precedes termination, any \texttt{TER}-type vector is also $\rho$-sufficient and thus treated like \texttt{SUF}-type in FNCR-LS.

One potential issue of FNCR-LS with $\T = 1$ is that, in the worst case, $\f{s}_k^{(1)} = -\alpha_k^{(0)}\f{g}_k$ may fail to be $\rho$-sufficient at every outer iteration $k$. In this situation, \cref{alg:cr_with_sd_si} always returns $\f{s}_k^{(1)}$ from line 14 as \texttt{INS}-type, causing FNCR-LS to degenerate into gradient descent with line-search. To prevent this, we set $\T > 1$, so that the sufficient descent check is skipped for $t < \T$. This ensures that, when \texttt{dtype} = \texttt{INS}, the returned vector $\f{s}_k^{(t)}$ is not parallel to the gradient, thereby avoiding reduction to a gradient method.

Furthermore, using the monotonic decreasing residual property \cref{eq:rt>rtp1} of CR, we can adapt the sufficiency parameter $\rho$ dynamically as
\begin{align}
\label{eq:beta_t}
\rho_t = \frac{\|\f{g}\|^2 }{\|\f{r}_k^{(t-1)}\|^2} \cdot \rho, \quad \rho_0 = \rho, \quad 0 < \rho < 1/2.
\end{align}
By monotonicity, $\rho_t$ increases with $t$, making the sufficiency condition progressively stricter. That is, unless $\f{s}_k^{(t)}$ provides a strong descent direction, FNCR-LS moves to the next outer iterate rather than spending additional computation on the current Newton system.

The  return types of \cref{alg:cr_with_sd_si} and their implications within \cref{alg:fncr-ls} are compactly given in \cref{tab:cr_return_types_compact}.
\begin{table}[htbp]
\centering
\caption{Return Types of \cref{alg:cr_with_sd_si} in FNCR-LS}
\begin{tabular}{|c|c|c|c|}
\hline
\textbf{dtype} & \textbf{Return} & \textbf{$\rho_t$-suff.} & \textbf{$\rho$-suff. / Notes} \\
\hline
\texttt{SUF} & $\f{s}_k^{(t-1)}, t>\T$ & Yes & Yes (by \cref{eq:rt>rtp1}) \\
\texttt{INS} & $\f{s}_k^{(t)}, t=\T$ & No & May or may not \\
\texttt{TER} & $\f{s}_k^{(t)}$ & $t=\TM$: Yes & $\|\f{r}_k\|\le \omega\|\f{g}_k\|$: Maybe; $t=\TM$: Yes \\
\hline
\end{tabular}
\label{tab:cr_return_types_compact}
\end{table}

Furthermore, when \cref{alg:cr_with_sd_si} terminates, the residual norm must satisfy one of the following: if \cref{alg:cr_with_sd_si} terminates due to the relative residual condition, then $\|\f{r}_k^{(t)}\| \leq \omega\|\f{g}_k\|$. Otherwise, termination occurs at some $t$ with $\T \leq t \leq \TM$, which implies
\begin{align}\label{eq:cr_convergence}
    \|\f{r}_k^{(t)}\| \leq \|\f{r}_k^{(\T)}\| \leq 2\left(\frac{\sqrt{\kappa} - 1}{\sqrt{\kappa} + 1}\right)^{\T}\|\f{g}_k\| \leq 2\exp\left(\frac{-2\T}{\sqrt{\kappa}}\right)\|\f{g}_k\|,
\end{align}
by the standard CR convergence rate \cite{saaditerative2003}. 

In essence, the primary difference between the classical inexact Newton's method and FNCR-LS lies in how they terminate CR. The classical method stops based on the relative residual or maximum iteration conditions, whereas FNCR-LS additionally enforces the sufficient descent condition. This simple algorithmic modification can lead to significant practical improvements (see \cref{sec:num_exp}). Moreover, with specific hyperparameter settings, FNCR-LS reduces to several familiar algorithms: with $0 < \omega < 1$ and $\T = d$, it becomes inexact or truncated Newton’s method; with $\omega = 0$ and $\T = d$, it reduces to damped Newton’s method; and with $\TM = 1$, it becomes CR-scaled gradient descent with line-search \cite{smee2025first}.
%

For clarity, the inner solver counter $t$ (i.e., $\f{s}_k^{(t)}$) can generally be omitted in our analyses; we retain only the outer solver counter $\f{s}_k$ without causing ambiguity. The inner counter will be reintroduced when necessary. We now present the following lemma, which will be used throughout the rest of this paper.
\begin{lemma}\label{lem:descent_eta_region}
Let $h > 0$, $0 < \rho < 1/2$, and $1 \leq t \leq d$. Suppose \cref{assmp:LH} holds, and let $\f{s}_k$ be a vector generated by \cref{alg:cr_with_sd_si} with $\f{H}_k$ and $\f{g}_k$, where $\f{H}_k$ is either 
\[
\f{H}_k = \He(\f{x}_k) \succeq h\f{I} \quad \text{or} \quad \f{H}_k = \He(\f{x}_k) + h\f{I} \succeq h\f{I}.
\] 
If the step-size $\eta$ satisfies
\begin{align}\label{eq:lem:descent_eta_region}
    0 \leq \eta \leq \min\Biggl\{1, \, \sqrt{\frac{3(1 - 2\rho)}{\LH}} \frac{h^{3/4}}{|\langle \f{g}_k,\f{s}_k \rangle|^{1/4}} \Biggr\},
\end{align}  
then $\eta \f{s}_k$ is $\rho$-sufficient, i.e., $f(\f{x}_k + \eta \f{s}_k) \leq f_k + \rho \eta \langle \f{g}_k, \f{s}_k \rangle$.
\end{lemma}
The proof follows standard arguments in the optimization literature and is deferred to \cref{sec:deferred_proofs}.
By \cref{lem:descent_eta_region}, \cref{alg:backtracking} guarantees to return a $\rho$-sufficient vector $\eta \f{s}_k$ with a step-size $\eta$ satisfying
\begin{align}\label{eq:backtracking_eta}
    0 < \min\Biggl\{1, \, \zeta \sqrt{\frac{3(1 - 2\rho)}{\LH}} \frac{h^{3/4}}{|\langle \f{g},\f{s}_k \rangle|^{1/4}} \Biggr\} \leq \eta.
\end{align}

Using these results, we next establish the iteration and operation complexities of FNCR-LS. We first consider the simpler strongly convex setting in \cref{sec:strongly_convex_functions}, followed by extensions to general convexity in \cref{sec:general_convex_functions}.

\subsection{Strongly Convex Setting}\label{sec:strongly_convex_functions}
In this section, we consider strongly convex functions, i.e., functions satisfying \cref{eq:convexity} with $\mu > 0$.  
Here the Hessian placeholder is the true Hessian, $\f{H}(\f{x}) = \He(\f{x})$, so that $\lambda_{\min} \geq \mu$.  
With the choice $\f{H}_k = \He(\f{x}_k)$, we can derive the following upper bound using the CR properties \cref{eq:sHr=0,eq:Hstp1>Hst}.
We note,
\begin{align*}
    \|\f{g}_k + \f{H}_k(\eta\f{s}_k)\| & = \|\f{g}_k + \f{H}_k\f{s}_k - (\f{H}_k\f{s}_k - \f{H}_k(\eta\f{s}_k))\| = \|-\f{r}_k - (1-\eta)\f{H}_k\f{s}_k\|\\
    & \op{\cref{eq:sHr=0}}{=} \sqrt{\|\f{r}_k\|^2 + (1-\eta)^2\|\f{H}_k\f{s}_k\|^2} \op{\cref{eq:Hstp1>Hst}}{\leq} \sqrt{\|\f{r}_k\|^2 + (1 - \eta)^2\|\f{g}_k\|^2}.
\end{align*}
By \cref{eq:assmp:LH:gkp1<p} and the triangle inequality, with $\f{d} = \eta\f{s}_k$, the above inequality implies the following bound,
\begin{align}\label{eq:assmp:LH:cr_bound}
    \|\f{g}_{k+1}\| \leq \frac{\LH\eta^2}{2}\|\f{s}_k\|^2 + \sqrt{\|\f{r}_k\|^2 + (1 - \eta)^2\|\f{g}_k\|^2}.
\end{align}

To derive global complexity guarantee, we first note that by the strong convexity of $f$, for any $\f{x} \in \F$, we have $\|\f{x} - \f{x}^*\| \leq \sqrt{2\delta(\f{x})/\mu} \leq \sqrt{2\delta_0/\mu}$, which implies the sub-level set $\F$ is bounded. 
Thus, our global complexity analysis can be restricted to the compact region defined by $\F$. Since $f$ is twice continuously differentiable, then there exists a constant $ H(\mathbf{x}_0) \triangleq H_0 > 0 $, depending on the initial point $\mathbf{x}_0$, such that $\|\f{H}(\f{x})\| \leq H_0$, for all $\f{x} \in \F$. In this case, the condition number is defined by $\kappa \triangleq H_0 / \mu$.
We are now ready to show the convergence of FNCR-LS under strongly convex settings. 
\begin{lemma}\label{lem:omega_fncr}
Let $f$ be $\mu$-strongly convex and \cref{assmp:LH} hold. Then FNCR-LS (\cref{alg:fncr-ls}) guarantees one of the following for each iteration $k$:\bigskip

\begin{subequations}\label{eq:lem:omega_fncr}
\noindent 
\begin{itemize}
\item \textbf{Case 1:} If $\f{s}_k$ is not $\rho$-sufficient and $\|\f{g}_{k+1}\| > 2\sqrt{2}\|\f{g}_k\|$,  
or if $\f{s}_k$ is $\rho$-sufficient and $\|\f{g}_{k+1}\| > 2\|\f{r}_k\|$, then
\begin{align}
    \delta_{k+1} &< \left(1 - \frac{\sqrt{2}\rho\mu^{3/2}}{\sqrt{2}\rho\mu^{3/2} + \LH \sqrt{\delta_{k+1}}}\right)\delta_k; \label{eq:lem:omega_fncr:1}
\end{align}

\item \textbf{Case 2:} If $\f{s}_k$ is not $\rho$-sufficient and $2\sqrt{2}\|\f{g}_k\| \geq \|\f{g}_{k+1}\|$, then
\begin{align}
    \delta_{k+1} &< \left(1 - \frac{3\zeta\rho(1-2\rho)\mu^{3/2}}{3\zeta\rho(1-2\rho)\mu^{3/2} + 2\LH\sqrt{\kappa}\sqrt{\delta_{k+1}}}\right)\delta_k; \label{eq:lem:omega_fncr:2}
\end{align}

\item \textbf{Case 3:} If $\f{s}_k$ is $\rho$-sufficient and $2\|\f{r}_k\| \geq \|\f{g}_{k+1}\|$, then
\begin{align}
    \delta_{k+1} &\leq \left(1 - \frac{\rho}{\rho + 2\kappa\omega^{2}}\right)\delta_k, \qquad \text{or} \label{eq:lem:omega_fncr:3} \\
    \delta_{k+1} &< \left(1 - \frac{\rho}{\rho + 8\kappa \exp\!\left(-4\T/\sqrt{\kappa}\right)}\right)\delta_k. \label{eq:lem:omega_fncr:4}
\end{align}
\end{itemize}
\end{subequations}
\end{lemma}

\begin{proof}
    From the characterization of the return types above, the reduction analysis can be divided into two cases, depending on whether the returned vector $\f{s}_k$ from \cref{alg:cr_with_sd_si} is $\rho$-sufficient or not. 
    If $\f{s}_k$ is not $\rho$-sufficient, from \cref{eq:assmp:LH:cr_bound}, we have,
    \begin{align*}
        \|\f{g}_{k+1}\| \leq \frac{\LH\eta^2}{2}\|\f{s}_k\|^2 + \sqrt{\|\f{r}_k\|^2 + (1-\eta)^2\|\f{g}_k\|^2} \op{(0 \, < \,\eta),\, \cref{eq:rt>rtp1}}{<} \frac{\LH\eta^2}{2}\|\f{s}_k\|^2 + \sqrt{2}\|\f{g}_k\|.
    \end{align*}
    If $\|\f{g}_{k+1}\| > 2\sqrt{2}\|\f{g}_k\|$, then $\LH\eta^2\|\f{s}_k\|^2 > \|\f{g}_{k+1}\|$ and,
    \begin{align*}
        \delta_{k+1} - \delta_k & \leq \rho\eta\lr{\f{g}_k,\f{s}_k} \op{\cref{eq:gst<-sHs}}{\leq} - \rho\mu\eta\|\f{s}_k\|^2 < -\frac{\rho\mu}{\LH\eta}\|\f{g}_{k+1}\| \op{(\eta \, < \, 1)}{<} - \frac{\rho\mu}{\LH}\sqrt{2\mu\delta_{k+1}}.
    \end{align*}
    With a few algebraic manipulations, we have
    \begin{align*}
        \delta_{k+1}\left(1 + \frac{\sqrt{2}\rho\mu^{3/2}}{\LH\sqrt{\delta_{k+1}}}\right) = \delta_{k+1} + \frac{\sqrt{2}\rho\mu^{3/2}\sqrt{\delta_{k+1}}}{\LH} & < \delta_k,
    \end{align*}
    which  implies \cref{eq:lem:omega_fncr:1}.
    We now consider the case, $2\sqrt{2}\|\f{g}_k\| \geq \|\f{g}_{k+1}\|$. By \emph{modus tollen}, \cref{lem:descent_eta_region} with $\eta = 1$ and $h = \mu$ implies, 
    \begin{align}\label{eq:sk_insufficient}
        \sqrt{\frac{3(1-2\rho)}{\LH}}\mu^{3/4} < |\lr{\f{g}_k,\f{s}_k}|^{1/4}.
    \end{align}
    By \cref{alg:backtracking}, a $\rho$-sufficient vector $\eta\f{s}_k$ will be returned, where $\eta$ satisfies \cref{eq:backtracking_eta}, and, hence,
    \begin{align*}
        \delta_{k+1} - \delta_k & \leq \rho\eta\lr{\f{g}_k,\f{s}_k} \op{\cref{eq:backtracking_eta}}{<} - \rho\zeta\sqrt{\frac{3(1-2\rho)\mu^{3/2}}{\LH}}\left|\lr{\f{g}_k,\f{s}_k}\right|^{3/4}\\
        & \op{\cref{eq:sk_insufficient}}{<} -\frac{3\zeta\rho(1 - 2\rho)\mu^{3/2}}{\LH}\left|\lr{\f{g}_k,\f{s}_k}\right|^{1/2} \op{\cref{eq:stp1g>stg},\cref{eq:lemma:alpha_ineq}}{\leq} -\frac{3\zeta\rho(1 - 2\rho)\mu^{3/2}}{\LH}\frac{\|\f{g}_{k}\|}{\sqrt{H_0}}\\
        & \leq -\frac{3\zeta\rho(1 - 2\rho)\mu}{\LH\sqrt{\kappa}}\frac{\|\f{g}_{k+1}\|}{2\sqrt{2}} < -\frac{3\zeta\rho(1 - 2\rho)\mu}{\LH\sqrt{\kappa}}\frac{\sqrt{2\mu\delta_{k+1}}}{2\sqrt{2}}.
    \end{align*}
    With a few algebraic manipulations, this implies \cref{eq:lem:omega_fncr:2}.
    Next, we consider the case where $\f{s}_k$ is $\rho$-sufficient. From \cref{eq:assmp:LH:cr_bound}, we now have, $\|\f{g}_{k+1}\| \leq \LH\|\f{s}_k\|^2/2 + \|\f{r}_k\|$.
    If $\|\f{g}_{k+1}\| > 2\|\f{r}_k\|$, then $\LH\|\f{s}_k\|^2 > \|\f{g}_{k+1}\|$ and
    \begin{align*}
        \delta_{k+1} - \delta_k & \leq \rho\lr{\f{g}_k,\f{s}_k} \op{\cref{eq:gst<-sHs}}{\leq} - \rho\mu\|\f{s}_k\|^2 < - \frac{\rho\mu}{\LH}\|\f{g}_{k+1}\| \leq - \frac{\rho\mu}{\LH}\sqrt{2\mu\delta_{k+1}},
    \end{align*}
    which implies \cref{eq:lem:omega_fncr:1}. 
    However, if $2\|\f{r}_k\| \geq \|\f{g}_{k+1}\|$, then by the discussion of the FNCR-LS return types above, \cref{alg:cr_with_sd_si} must return $\f{s}_k$ satisfying either $\|\f{r}_k\| \leq \omega\|\f{g}_k\|$ or $\|\f{r}_k\| \leq 2\exp(-2\T/\sqrt{\kappa})\|\f{g}_k\|$. 
    These imply, respectively, 
    \[
    2\omega\|\f{g}_k\| \geq \|\f{g}_{k+1}\| 
    \quad \text{and} \quad 
    4\exp(-2\T/\sqrt{\kappa})\|\f{g}_k\| \geq \|\f{g}_{k+1}\|.
    \]
    In the former case, we have
    \begin{align*}
        \delta_{k+1} - \delta_k 
        &\leq \rho\langle \f{g}_k,\f{s}_k\rangle 
        \overset{\cref{eq:stp1g>stg},\,\cref{eq:lemma:alpha_ineq}}{\leq} 
        - \frac{\rho}{H_0}\|\f{g}_k\|^2 
        \leq - \frac{\rho}{4H_0\omega^2}\|\f{g}_{k+1}\|^2 
        \leq - \frac{\rho\mu}{2H_0\omega^2}\delta_{k+1},
    \end{align*}
    which implies \cref{eq:lem:omega_fncr:3}.  
    Otherwise, \cref{eq:lem:omega_fncr:4} follows by a similar argument.
\end{proof}

As a sanity check to ensure \cref{lem:omega_fncr} aligns with our understanding of Newton’s method, consider a quadratic problem $f(\f{x}) = 0.5\langle\f{x},\f{Ax}\rangle - \langle\f{b},\f{x}\rangle$ with $\LH = 0$. Classical Newton's method converges in one step, and in this case \cref{eq:lem:omega_fncr:1,eq:lem:omega_fncr:2} also imply one-step convergence, leaving only \cref{eq:lem:omega_fncr:3,eq:lem:omega_fncr:4}. With $\T = \TM = d$ and $\omega = 0$, \cref{alg:cr_with_sd_si} terminates by the inexactness condition, returning $\f{s}_k = -\f{H}_k^{-1}\f{g}_k$ and thus yielding \cref{eq:lem:omega_fncr:3}. Hence, FNCR-LS converges in a single iteration, consistent with Newton’s method.

Next, we derive the overall iteration complexity of FNCR-LS from \cref{lem:omega_fncr}. 
Suppose FNCR-LS runs for $k$ iterations, and define
\begin{align*}
    \mathcal{A}_k \triangleq \bigl\{\, i \leq k \;\big|\; \text{iteration $i$ satisfies \cref{eq:lem:omega_fncr:1} or \cref{eq:lem:omega_fncr:2}} \,\bigr\}.
\end{align*}
First, we note from \cref{lem:omega_fncr} that $\delta_{i+1} < \delta_i$ holds for every iteration, regardless of whether $i \in \mathcal{A}_k$. 
So, by \cref{eq:lem:omega_fncr}, for every iteration $k$ we obtain $\delta_{k+1} \leq c_0\delta_k$, where $c_0 \triangleq 1 - (1 + \max\{\kappa c_1,c_2\})^{-1}$ with
\begin{align*}
c_1 &\triangleq \frac{1}{\rho}\max\left\{2\omega^2,\,8\exp(-4\T/\sqrt{\kappa})\right\}, \\
c_2 &\triangleq \frac{\LH\,\sqrt{\delta_{0}}}{\rho\,\mu^{3/2}}\,
\max\!\left\{\,\frac{1}{\sqrt{2}},\; \frac{2\sqrt{\kappa}}{3\zeta\,(1-2\rho)}\,\right\}.
\end{align*}
Now, for each $i \in \mathcal{A}_k$, we have
\begin{align*}
    \delta_{i+1} & \op{\cref{eq:lem:omega_fncr:1},\,\cref{eq:lem:omega_fncr:2}}{\leq} \left(1 - \frac{1}{1 + c_2\sqrt{\delta_{i+1}}}\right)\delta_i \leq \left(1 - \frac{1}{1 + c_2c_0^{i/2}\sqrt{\delta_{0}}}\right)\delta_i.
\end{align*}
In particular, if $|\mathcal{A}_k| \geq k/2$, then FNCR-LS achieves a superlinear convergence,
\begin{align*}
    \delta_{k+1} & \op{w.l.o.g}{\leq} \left(1 - \frac{1}{1 + c_2c_0^{k/2}\sqrt{\delta_{0}}}\right)\delta_k \leq \prod_{i\in\mathcal{A}_k}\left(1 - \frac{1}{1 + c_2c_0^{i/2}\sqrt{\delta_{0}}}\right)\delta_0\\
    & \leq \prod_{i=0}^{|\mathcal{A}_k|}\left(1 - \frac{1}{1 + c_2c_0^{i/2}\sqrt{\delta_{0}}}\right)\delta_0 \leq \prod_{i=0}^{k/2}\left(1 - \frac{1}{1 + c_2c_0^{i/2}\sqrt{\delta_{0}}}\right)\delta_0.
\end{align*}
Otherwise, $|\mathcal{A}_k^\complement| \geq k/2$, FNCR-LS achieves a linear convergence,
\begin{align*}
    \delta_{k+1} & \leq \left(1 - \frac{1}{1 + \kappa c_1}\right)\delta_k \leq \left(1 - \frac{1}{1 + \kappa c_1}\right)^{|\mathcal{A}_k^\complement|}\delta_0 \leq \left(1 - \frac{1}{1 + \kappa c_1}\right)^{k/2}\delta_0.
\end{align*}
In particular, setting $0 < \omega \leq \sqrt{(2\kappa)^{-1}}$ and $\T \geq \lceil\sqrt{\kappa}\ln(8\kappa)/4\rceil$, we obtain
\begin{align*}
    c_1 = \frac{1}{\rho}\max\left\{2\omega^2, \, 8\exp(-4\T\sqrt{\kappa^{-1}})\right\} \leq \frac{1}{\rho}\max\left\{\frac{1}{\kappa}, \, 8\exp(-\ln(8\kappa))\right\} = \frac{1}{\rho\kappa},
\end{align*}
which yields the following \textit{problem-independent} linear convergence rate,
\begin{align*}
    \delta_{k+1} \leq \frac{\delta_0}{(1 + \rho)^{k/2}}.
\end{align*}

Putting all together, we have the following complexity result.
\begin{corollary}\label{cor:superlinear_linear_fncr-ls}
Let $0 < \omega \leq \sqrt{(2\kappa)^{-1}}$ and $\T \geq \lceil \sqrt{\kappa} \ln(8\kappa)/4 \rceil$. Then FNCR-LS (\cref{alg:fncr-ls}) achieves either superlinear convergence or a linear convergence that is independent of the condition number. 
In particular, for the linear case, the iteration and operation complexities to achieve $\varepsilon$-suboptimality are  
$\mathcal{O}(\ln(\varepsilon^{-1}))$ and 
$\mathcal{O}(\sqrt{\kappa} \ln(\kappa) \ln(\varepsilon^{-1}))$,   respectively.
\end{corollary}

The superlinear convergence of FNCR-LS depends on the quantity $|\mathcal{A}_k|$. Nonetheless, there are two ways to guarantee overall superlinear convergence. 
First, observe that for the linear reductions \cref{eq:lem:omega_fncr:3,eq:lem:omega_fncr:4} to hold, $\f{s}_k$ must be $\rho$-sufficient and $\|\f{g}_{k+1}\| \leq 2\|\f{r}_k\|$. 
By setting $\omega = 0$ and $\TM = d$, \cref{alg:cr_with_sd_si} always returns $\|\f{r}_k\| = 0$. 
Thus, whenever the linear reduction holds, we must also have $\|\f{g}_{k+1}\| = 0$, which by strong convexity implies $\delta_{k+1} = 0$. 
Hence, $\mathcal{A}_k^{\complement}$ can occur at most once (i.e., $|\mathcal{A}_k^{\complement}| = 1$), ensuring overall superlinear convergence. 
The second way is to convert the reductions \cref{eq:lem:omega_fncr:3,eq:lem:omega_fncr:4} into a form similar to \cref{eq:lem:omega_fncr:1,eq:lem:omega_fncr:2}. 
Let $\omega \leq \|\f{g}_k\|$ and $\T \geq \lceil \sqrt{\kappa}\ln(2\|\f{g}_k\|^{-1})/4 \rceil$. Then $\|\f{g}_{k+1}\| \leq 2\|\f{r}_k\| \leq 2\|\f{g}_k\|^2$, and since $\f{s}_k$ is $\rho$-sufficient, we have
\begin{align*}
    \sqrt{2\mu\delta_{k+1}} \leq \|\f{g}_{k+1}\| \leq 2\|\f{g}_k\|^2 \leq \frac{2H_0}{\rho}(\delta_k - \delta_{k+1}),
\end{align*}
which transforms both \cref{eq:lem:omega_fncr:3,eq:lem:omega_fncr:4} into
\begin{align*}
    \delta_{k+1} < \left(1 - \frac{\rho}{\rho + \sqrt{2\kappa H_0}\sqrt{\delta_{k+1}}}\right)\delta_k.
\end{align*}
Using a similar argument as above, we obtain the following result:

\begin{corollary}\label{cor:superlinear}
Let $\omega \leq \|\f{g}_k\|$ and $\T \geq \lceil \sqrt{\kappa}\ln(2\|\f{g}_k\|^{-1})/4 \rceil$. FNCR-LS (\cref{alg:fncr-ls}) achieves superlinear convergence in terms of iteration complexity.
\end{corollary}

\begin{remark}
    For $\T = \TM = d$, FNCR-LS reduces to the damped Newton's method when $\omega = 0$, and to the inexact Newton's method for general $\omega$. 
    In both cases, our analysis shows that these methods can achieve overall superlinear convergence, provided that $\T = d$. 
    This dispels the longstanding view that Newton-type methods necessarily exhibit worse worst-case convergence behavior than gradient methods. 
\end{remark}

Note that the convergence in \cref{cor:superlinear}, while superlinear, is still consistent with the lower bounds for second-order methods under strong convexity and Hessian Lipschitz smoothness established in \cite{arjevani2019oracle}. 
A direct analytic comparison is difficult: our result in \cref{cor:superlinear} characterises a superlinear rate, whereas the lower bound reflects two distinct phases, namely an initial constant reduction in the damped Newton phase, followed by local quadratic convergence in the pure Newton phase. 
A more practical approach is to substitute numerical values for the constants and empirically compare the number of iterations required to reach a prescribed accuracy. 
Across a wide range of tested values, the worst-case superlinear convergence of FNCR-LS is consistently slower than the lower bound.\footnote{For example, with $H_0 = 10^{6}$, $\LH = 100$, $\mu = 5$, $\delta_0 = 25{,}000$, $\|\f{x}_0 - \f{x}^*\| = 100$, and $\varepsilon = 10^{-14}$, the lower bound requires 5 iterations, whereas FNCR-LS requires more than $10^{6}$ iterations under its superlinear rate.}

\cref{cor:superlinear} established a global superlinear convergence rate for FNCR-LS. 
However, similarly to the classical exact Newton's method, FNCR-LS can achieve a  \emph{local quadratic} convergence rate, even when using only inexact updates.

\begin{theorem}[Local Quadratic Convergence with Inexact Newton's Step]\label{thm:fncr:local:quadratic}
Let $1/3 < \rho < 1/2$, 
\[
0 \leq \omega < \frac{\|\f{g}_k\|\LH}{2\mu^2} \quad \text{and} \quad \T \geq \Big\lceil \frac{\sqrt{\kappa}}{2}\ln\Big(\frac{4\LH}{\|\f{g}_k\|\mu^2}\Big) \Big\rceil.
\] 
Let $\f{x}_{0}$ satisfy $\|\f{g}_0\| \leq r$ where
\begin{align}\label{eq:r}
    r \triangleq \frac{3(1-2\rho)\mu^2}{\LH}.
\end{align}
Then FNCR-LS (\cref{alg:fncr-ls}) achieves a quadratic convergence rate:
\begin{align}\label{eq:thm:fncr:local:quadratic}
    \delta_{k+1} & < \left(\frac{\mu^2}{\LH}\right)^2 \frac{3(1-2\rho)^{2^{k+1}}}{2\mu},
\end{align}
implying in  at most $\mathcal{O}(\ln(\ln(\mu^3/(\LH\varepsilon))))$ iterations, it achieves $\varepsilon$-suboptimality.
\end{theorem}


\begin{proof}
We first note that, in this local region, the step $\f{s}_k$ is always $\rho$-sufficient. Indeed, 
\begin{align*}
    |\lr{\f{g}_k,\f{s}_k}| 
    \op{\cref{eq:stp1g>stg}}{\leq} \lr{\f{g}_k,\f{H}_k^{-1}\f{g}_k} 
    \leq \frac{1}{\mu}\|\f{g}_k\|^2 
    \leq \left(\frac{3(1-2\rho)}{\LH}\right)^{2}\mu^3,
\end{align*}
which implies 
\[
1 \leq \sqrt{\frac{3(1-2\rho)}{\LH}}\;\mu^{3/4}|\lr{\f{g}_k,\f{s}_k}|^{-1/4}.
\]
By \cref{lem:descent_eta_region}, it follows that $\f{s}_k$ is $\rho$-sufficient,  
%
Consequently, under the conditions on $\omega$ and $\T$, the residual necessarily satisfies
\[
\|\f{r}_k\| \leq \frac{\LH}{2\mu^2}\|\f{g}_k\|^2.
\]
Applying \cref{eq:assmp:LH:cr_bound} gives
\begin{align*}
    \|\f{g}_{k+1}\| &\leq \frac{\LH}{2}\|\f{s}_k\|^2 + \|\f{r}_k\| 
    \op{\cref{eq:stp1>st}}{\leq} \frac{\LH}{\mu^2}\|\f{g}_k\|^2.
\end{align*}
Since the above inequality holds for all $k$, this implies
\begin{align*}
    \sqrt{2\mu\delta_{k+1}} &\op{\cref{eq:convexity}}{\leq} \|\f{g}_{k+1}\| 
    < \frac{\LH}{\mu^2}\|\f{g}_k\|^2 
    \leq \left(\frac{\LH}{\mu^2}\right)^{2^k-1} \|\f{g}_0\|^{2^k} 
    \leq \frac{\mu^2}{\LH} (3(1-2\rho))^{2^k}.
\end{align*}
With $1/3 < \rho < 1/2$, the result \eqref{eq:thm:fncr:local:quadratic} establishes quadratic convergence.
\end{proof}

\begin{remark}
To the best of our knowledge, \cref{thm:fncr:local:quadratic} is the first result showing that Newton's method, even with inexact Newton steps, can achieve a local quadratic convergence rate.
\end{remark}

\subsection{Extensions to General Convexity}\label{sec:general_convex_functions}
We now go beyond the strongly convex setting and extend FNCR-LS to a broader class of convex functions, where $\mu = 0$ in \cref{eq:convexity}. 
In the strongly convex case, the compactness of the sub-level set $\F$ followed directly from strong convexity. 
In the general convex setting, however, we introduce compactness explicitly as an assumption, following the standard practice in related works~\cite{mishchenko2023regularized,doikov2024gradient,nesterov2006cubic,doikov2024super}.
\begin{assumption}[Bounded Sub-level Set]\label{assmp:boundedness}
The sub-level set $\F$ is compact. More precisely, for any given $\f{x}_{0}$, there exists $0 \leq D < \infty$ such that
\[
\|\f{x} - \f{x}^*\| \leq D, \quad \forall \f{x} \in \F.
\]
\end{assumption}
Together with convexity, \cref{assmp:boundedness} yields the following bound
\begin{align}\label{eq:assmp:boundedess_convexity}
    \delta(\f{x}) & \op{\cref{eq:convexity}}{\leq} \langle \f{g}(\f{x}),\, \f{x} - \f{x}^* \rangle 
    \leq D \|\f{g}(\f{x})\|, \quad \forall \f{x} \in \F.
\end{align}
Furthermore, we employ the gradient-regularized Hessian
\begin{align}
\label{eq:H_reg}    
\f{H}(\f{x}) = \He(\f{x}) + \sigma \sqrt{\|\f{g}(\f{x})\|}, \quad \sigma > 0,
\end{align}
which we will simply denote by $\f{H}$.  
For the remainder of the paper, we refer to this variant of the FN framework as \emph{FNCR-reg-LS}, that is, \cref{alg:fncr-ls} combined with \cref{alg:cr_with_sd_si} and the gradient-regularized Hessian.

To establish the convergence rate of FNCR-reg-LS, we divide the analysis into three stages. 
First, we introduce foundational lemmas (\cref{lem:obj_descents,lem:sufficient_inexactness}), which are adapted from standard arguments in worst-case complexity analysis; their proofs are deferred to \cref{sec:deferred_proofs}. 
Next, we quantify the reductions achieved by \texttt{TER}-type (\cref{lem:reg:sol}), \texttt{INS}-type (\cref{lem:reg:unf}), and \texttt{SUF}-type (\cref{lem:reg:fth}) directions. 
Finally, combining these results, we establish the overall iteration complexity of FNCR-reg-LS in \cref{thm:reg_fncr}.

\begin{lemma}\label{lem:sufficient_inexactness}
    Suppose $f$ is convex and \cref{assmp:LH} holds. 
    Let $\sigma > 0$, $q_1 > 0$, $0 \leq q_2 < 1$, and $0 < \rho < 1/2$. 
    Assume $\|\f{r}_k\| \leq q_2 \|\f{g}_k\|$ and $\f{s}_k$ being $\rho$-sufficient. Then:
    \begin{itemize}
        \item If $\|\f{s}_k\|^2 \geq 2q_1 \|\f{g}_k\| / \LH$, then
        \[
            f(\f{x}_k + \f{s}_k) \;\leq\; f_k - \frac{2\rho\sigma q_1}{\LH} \|\f{g}_k\|^{3/2}.
        \]
        \item Otherwise,
        \begin{align*}
            \|\f{g}(\f{x}_k + \f{s}_k)\| \leq \left(q_1 + \sigma \sqrt{\frac{2q_1}{\LH}} + q_2\right)\|\f{g}_k\|, \quad \text{and} \quad f(\f{x}_k + \f{s}_k) < f_k.
        \end{align*}
    \end{itemize}
\end{lemma}

\begin{lemma}\label{lem:obj_descents}
    Suppose $f$ is convex and \cref{assmp:LH} holds. Let $\sigma > 0$, $0 < \rho < 1/2$, and $1 \leq t \leq d$. 
    If $\f{s}_k$ is not $\rho$-sufficient, then \cref{alg:backtracking} returns a scaled step $\eta \f{s}_k$ that is $\rho$-sufficient and satisfies
    \begin{align*}
        f(\f{x}_k + \eta\f{s}_k) 
        &< f_k - \rho \zeta \left(\frac{3(1 - 2\rho)}{\LH}\right)^{2} \sigma^{3} \|\f{g}_k\|^{3/2}.
    \end{align*}
\end{lemma}

The above lemmas form the basis for analyzing the three types of return directions: \texttt{TER} (\cref{lem:reg:sol}), \texttt{INS} (\cref{lem:reg:unf}), and \texttt{SUF} (\cref{lem:reg:fth}).

\begin{lemma}[\texttt{TER}-type Reduction]\label{lem:reg:sol}
Suppose $f$ is convex and \cref{assmp:LH} holds. Let $\sigma > 0$, $0 < \rho < 1/2$, $0 < \zeta < 1$,
\[
0 \leq \omega < \frac{\sigma^2}{2(\LH + 6\sigma^2)} \quad \text{and} \quad \TM \geq \Big\lceil \frac{\sqrt{\kappa}}{2}\ln\Big(\frac{4(\LH + 6\sigma^2)}{\sigma^2}\Big) \Big\rceil.
\] 
With FNCR-reg-LS (\cref{alg:fncr-ls} with $\f{H}_k$ as in \cref{eq:H_reg}), if \cref{alg:cr_with_sd_si} returns \texttt{dtype} = \texttt{TER}, one of the following holds:
\begin{align}
    \delta_{k+1} &< \delta_k - \rho\zeta\left(\frac{3(1-2\rho)}{\LH}\right)^2\sigma^3\|\f{g}_k\|^{3/2}, \label{eq:lem:reg:sol:1} \\
    \delta_{k+1} &< \delta_k - \frac{\rho\sigma^3}{6(\LH + 6\sigma^2)^2}\|\f{g}_k\|^{3/2}, \label{eq:lem:reg:sol:2} \\
    \|\f{g}_{k+1}\| &< \frac{\sigma^2}{\LH + 6\sigma^2}\|\f{g}_k\|, \quad \text{and} \quad \delta_{k+1} < \delta_k. \label{eq:lem:reg:sol:3}
\end{align}
\end{lemma}

\begin{proof}
When \cref{alg:cr_with_sd_si} returns \texttt{dtype} = \texttt{TER}, the vector $\f{s}_k$ satisfies either $\|\f{r}_k\| \leq \omega \|\f{g}_k\|$ (relative residual condition), or 
$ \|\f{r}_k\| \leq 2\exp(-2\TM\sqrt{\kappa}^{-1})\|\f{g}_k\|$ (maximum iteration condition). 
In either case, the bounds on $\omega$ and $\TM$ given in the lemma ensure that the residual satisfies 
\[
\|\f{r}_k^{(t)}\| \leq \frac{\sigma^2}{2(\LH + 6\sigma^2)} \|\f{g}_k\|,
\]
and FNCR-reg-LS enters the backtracking line-search (\cref{alg:backtracking}).
If $\f{s}_k$ is not $\rho$-sufficient, the line-search scales $\f{s}_k$ to $\f{d}_k = \eta \f{s}_k$ for some $\eta < 1$, and by \cref{lem:obj_descents}, \cref{eq:lem:reg:sol:1} holds. 
Otherwise, $\f{s}_k$ is $\rho$-sufficient, so $\f{d}_k = \f{s}_k$. Applying \cref{lem:sufficient_inexactness} with 
\[
q_1 = \frac{\LH\sigma^2}{12(\LH + 6\sigma^2)^2}, \quad q_2 = \frac{\sigma^2}{2(\LH + 6\sigma^2)},
\] 
yields \cref{eq:lem:reg:sol:2} and
\begin{align*}
\|\f{g}_{k+1}\| &\leq \left(q_1 + \sigma \sqrt{\frac{2q_1}{\LH}} + q_2\right)\|\f{g}_k\| 
< \frac{\sigma^2}{\LH + 6\sigma^2}\|\f{g}_k\|,
\end{align*}
which establishes \cref{eq:lem:reg:sol:3}.
\end{proof}

\begin{lemma}[\texttt{INS}-type Reduction]\label{lem:reg:unf}
    Suppose $f$ is convex and \cref{assmp:LH} holds. 
    Let $\sigma > 0$, $0 < \rho < \min\{1/2, \sigma^4/(16(\LH + 6\sigma^2)^2)\}$, and $0 < \zeta < 1$. 
    With FNCR-reg-LS (\cref{alg:fncr-ls} using $\f{H}_k$ from \cref{eq:H_reg}), if \cref{alg:cr_with_sd_si} returns \texttt{dtype} = \texttt{INS}, then one of the following holds:
    \begin{align}
        f_{k+1} &< f_k - \zeta\left(\frac{3}{4\LH}\right)^2\sigma^{3}\|\f{g}_k\|^{3/2}, \label{eq:lem:reg:unf:1}
    \end{align}
    or one of \cref{eq:lem:reg:sol:1,eq:lem:reg:sol:2,eq:lem:reg:sol:3} from \cref{lem:reg:sol}.
\end{lemma}

\begin{proof}
    When \cref{alg:cr_with_sd_si} returns \texttt{dtype} = \texttt{INS}, FNCR-reg-LS enters the backtracking line-search routine. 
    We distinguish two cases depending on $\rho_{\T}$.  
    Suppose $\rho_{\T} < 1/4$.  
    In this case, $\f{s}_k^{(\T)}$ is not $\rho_{\T}$-sufficient and hence also not $1/4$-sufficient. 
    Therefore, by \cref{lem:obj_descents} with $\rho = 1/4$, we obtain \cref{eq:lem:reg:unf:1}. 
    Otherwise, if $\rho_{\T} \geq 1/4$, the analysis is similar to that of \cref{lem:reg:sol}. 
    By the construction of the adaptive parameter $\rho_{\T}$ in \cref{eq:beta_t}, we have
    \begin{align*}
        \|\f{r}_k^{(\T)}\| 
        &< \|\f{r}_k^{(\T-1)}\| 
        \leq 2\sqrt{\rho}\|\f{g}_k\| 
        \leq \frac{\sigma^2}{2(\LH + 6\sigma^2)}\|\f{g}_k\|,
    \end{align*}
    where the last inequality uses the upper bound on $\rho$. 
    
    Now, if $\f{s}_k^{(\T)}$ is not $\rho$-sufficient, then \cref{lem:obj_descents} yields \cref{eq:lem:reg:sol:1}. 
    Otherwise, $\f{s}_k^{(\T)}$ is $\rho$-sufficient, and applying \cref{lem:sufficient_inexactness} $q_1 = \LH\sigma^2 / (12(\LH + 6\sigma^2)^2)$ and $q_2 = \sigma^2/(2(\LH + 6\sigma^2))$ gives 
    \cref{eq:lem:reg:sol:2,eq:lem:reg:sol:3}.
\end{proof}

\begin{lemma}[\texttt{SUF}-type Reduction]\label{lem:reg:fth}
    Suppose $f$ is convex and \cref{assmp:LH,assmp:boundedness} hold. 
    Let $0 < \rho < 1/2$, $\theta > 0$, $\sigma > 0$, and 
    $\T \geq \lceil \theta / \min\{1, \sqrt{\|\f{g}_k\|}\} \rceil$. 
    With FNCR-reg-LS (\cref{alg:fncr-ls} using $\f{H}_k$ from \cref{eq:H_reg}), if \cref{alg:cr_with_sd_si} returns \texttt{dtype} = \texttt{SUF}, 
    then there exists $0 \leq H_0 < \infty$ (depending on $\f{x}_0$) such that
    \begin{align}\label{eq:lem:reg:fth}
        f_{k+1} < f_k - \frac{\theta \rho}{H_0 + \sigma}\|\f{g}_k\|^{3/2}.
    \end{align}
\end{lemma}

\begin{proof}
    By \cref{assmp:boundedness}, there exists $H_0 > 0$ (depending on $\f{x}_0$) such that, for all $\f{x} \in \F$,  
    \[
        \|\f{H}(\f{x})\| < H_0 + \sigma \sqrt{\|\f{g}(\f{x})\|}.
    \]
    Since FNCR-reg-LS is a descent algorithm, we always have $\f{x}_k \in \F$.  
    Moreover, \texttt{dtype} = \texttt{SUF} implies $t \geq \T$ and $\f{s}_k^{(t)}$ is $\rho_t$-sufficient.  
    Using the monotonicity of $\rho_t$ and $\langle \f{g}_k, \f{s}_k^{(t)}\rangle$ from 
    \cref{eq:rt>rtp1,eq:stp1g>stg,eq:beta_t}, we obtain
    \begin{align*}
        f_{k+1} - f_k  \leq \rho_t\lr{\f{g}_k,\f{s}_k^{(t)}} &\leq \rho_\T\lr{\f{g}_k, \f{s}_k^{(\T)}} \\
        &\op{\cref{eq:beta_t},\cref{eq:lem:gst<gstm1-lr_weaker}}{\leq} \rho\frac{\|\f{g}_{k}\|^2}{\|\f{r}_k^{(\T-1)}\|^2}\left(\lr{\f{g}_{k},\f{s}_k^{(\T-1)}} - \frac{\|\f{r}_k^{(\T-1)}\|^2}{H_0 + \sigma\sqrt{\|\f{g}_k\|}}\right)\\
        & \op{\cref{eq:rt>rtp1}}{<} \rho\|\f{g}_{k}\|^2\left(\frac{\lr{\f{g}_{k},\f{s}_k^{(\T-1)}}}{\|\f{r}_k^{(\T-2)}\|^2} - \frac{1}{H_0 + \sigma\sqrt{\|\f{g}_k\|}}\right) \\
        &\op{\cref{eq:beta_t}}{=} \rho_{T-1}\lr{\f{g}_{k},\f{s}_k^{(\T-1)}} - \frac{\rho\|\f{g}_{k}\|^2}{H_0 + \sigma\sqrt{\|\f{g}_k\|}} \\
        & < \cdots < \rho\|\f{g}_{k}\|^2\left(\frac{\lr{\f{g}_{k},\f{s}_k^{(1)}}}{\|\f{r}_k^{(0)}\|^2} - \frac{\T-1}{H_0 + \sigma\sqrt{\|\f{g}_k\|}}\right)\\
        & = \rho\|\f{g}_{k}\|^2\left(-\alpha_0 - \frac{\T - 1}{H_0 + \sigma\sqrt{\|\f{g}_k\|}}\right) \op{\cref{eq:lemma:alpha_ineq}}{\leq} - \frac{\T \rho \|\f{g}_{k}\|^2}{H_0 +\sigma\sqrt{\|\f{g}_k\|}}.
    \end{align*}
    Now, if $\|\f{g}_k\| \geq 1$, then $\T \geq \lceil\theta\rceil$ and
    \begin{align*}
        f_{k+1} - f_k & < - \frac{\T \rho \|\f{g}_{k}\|^2}{H_0 +\sigma\sqrt{\|\f{g}_k\|}} \leq - \frac{\theta\rho}{H_0/\sqrt{\|\f{g}_k\|} +\sigma}\|\f{g}_{k}\|^{3/2} \leq - \frac{\theta\rho}{H_0 +\sigma}\|\f{g}_{k}\|^{3/2}.
    \end{align*}
    Otherwise, if $\|\f{g}_k\| < 1$, then $\T \geq \lceil \theta/\sqrt{\|\f{g}_k\|}\rceil$, which gives
    \[
        f_{k+1} - f_k 
        < - \frac{\T \rho \|\f{g}_k\|^2}{H_0 + \sigma\sqrt{\|\f{g}_k\|}}
        \leq - \frac{\theta\rho}{H_0 + \sigma}\|\f{g}_k\|^{3/2}.
    \]
    Thus, in both cases, the claimed bound \cref{eq:lem:reg:fth} holds.
\end{proof}


\begin{theorem}[Global Convergence of FNCR-reg-LS]\label{thm:reg_fncr}
    Suppose $f$ is convex and \cref{assmp:LH,assmp:boundedness} hold. 
    Let $\sigma > 0$, $0 < \rho < \min\{1/2, \, \sigma^4/(16(\LH + 6\sigma^2)^2)\}$, 
    $0 \leq \omega < \sigma^2/(\LH + 6\sigma^2)$, and $\theta > 0$. 
    Further, let $\lceil \theta / \min\{1,\sqrt{\|\f{g}_k\|}\}\rceil \leq \T \leq \TM \leq d$. 
    Then FNCR-reg-LS (\cref{alg:fncr-ls} using $\f{H}_k$ from \cref{eq:H_reg}) has a worst-case iteration complexity $\mathcal{O}(1/\sqrt{\varepsilon})$.
\end{theorem}

\begin{proof}
    Recall that we denote the function value gap by $\delta_k \triangleq f_k - f^*$. 
    For the iteration $k$ of FNCR-reg-LS, define 
    \begin{align*}
        \mathcal{C}_k \triangleq \{i \leq k \,|\, \text{iteration $i$ gives the relation } \cref{eq:lem:reg:sol:3}\},
    \end{align*}
    with $\mathcal{C}_{k}^{\complement}$ denoting its complement. 
    We analyze separately the cases $|\mathcal{C}_k| > k/2$ and $|\mathcal{C}_k| \leq k/2$.
    \bigskip

    \noindent \textbf{Case 1:} $|\mathcal{C}_k| > k/2$.  
    Recall the update direction is $\f{d}_k = \eta\f{s}_k$, where $0 < \eta \leq 1$ and $\f{s}_k$ is returned by \cref{alg:cr_with_sd_si}. We have
    \begin{align*}
        \|\f{g}_{k+1} - \f{g}_k - \f{H}_k\f{d}_k + \sigma\sqrt{\|\f{g}_k\|}\f{d}_k\| 
        &= \|\f{g}_{k+1} - \f{g}_k - \He(\f{x}_k)\f{d}_k\| \op{\cref{eq:assmp:LH:gkp1<p}}{\leq} \frac{\LH}{2}\|\f{d}_k\|^2,
    \end{align*}
    which implies
    \begin{align*}
        \|\f{g}_{k+1}\| 
        & \leq \frac{\LH}{2}\|\f{d}_k\|^2 + \|\f{g}_k\| + \|\f{H}_k\f{d}_k\| + \sigma\sqrt{\|\f{g}_k\|}\|\f{d}_k\| \\
        & \op{\eta \leq 1}{\leq} \frac{\LH}{2}\|\f{s}_k\|^2 + \|\f{g}_k\| + \|\f{H}_k\f{s}_k\| + \sigma\sqrt{\|\f{g}_k\|}\|\f{s}_k\| \\
        & \op{\cref{eq:stp1>st},\cref{eq:Hstp1>Hst}}{\leq} \frac{\LH}{2\sigma^2}\|\f{g}_k\| + 3\|\f{g}_k\| 
        = \frac{\LH + 6\sigma^2}{2\sigma^2}\|\f{g}_k\|,
    \end{align*}
    where we have used the fact that $$ \|\f{s}_k\|\leq\|(\nabla^2 f(\mathbf{x}) + \sigma\sqrt{\|\mathbf{g}_k\|}\mathbf{I})^{-1}\mathbf{g}_k\| \leq \|\mathbf{g}_k\|/(\sigma\sqrt{\|\mathbf{g}_k\|}).$$
    Define $L \triangleq (\LH + 6\sigma^2)/(2\sigma^2) \geq 1$.  
    Then for each $i \leq k$, if $i \in \mathcal{C}_k$, we have $\|\f{g}_{i+1}\| \leq \|\f{g}_i\|/(2L)$; otherwise, $\|\f{g}_{i+1}\| \leq L\|\f{g}_i\|$.  
    Hence,
    \begin{align}
        \label{eq:case_1}
        \delta_k / D &\op{\cref{eq:assmp:boundedess_convexity}}{\leq} \|\f{g}_{k}\|  \leq  \frac{L^{|\mathcal{C}_{k}^{\complement}|}}{(2L)^{|\mathcal{C}_k|}} \|\f{g}_0\| = \frac{L^{k - |\mathcal{C}_k|}}{(2L)^{|\mathcal{C}_k|}} \|\f{g}_0\| = \frac{L^{k - 2|\mathcal{C}_k|}}{2^{|\mathcal{C}_k|}} \|\f{g}_0\| \leq \frac{1}{2^{k/2}}\|\f{g}_0\|. 
    \end{align}

    \noindent \textbf{Case 2:} $|\mathcal{C}_k| \leq k/2$.  
    Then $|\mathcal{C}_{k}^{\complement}| > k/2$.  
    By \cref{lem:reg:fth,lem:reg:sol,lem:reg:unf}, for each $i \in \mathcal{C}_{k}^{\complement}$, we obtain
    \[
        f_{i+1} < f_i - C\|\f{g}_i\|^{3/2},
    \]
    where
    \[
        C \triangleq \rho\sigma^3 \min\!\left\{\zeta\left(\frac{3(1-2\rho)}{\LH}\right)^2, \ \frac{1}{6(\LH + 6\sigma^2)^2}, \ \frac{\zeta}{\rho}\left(\frac{3}{4\LH}\right)^2, \ \frac{\theta}{(H_0 + \sigma)\sigma^3}\right\}.
    \]
    Now, following \cite{nesterov2006cubic}, we have
     \begin{align*}
        \frac{1}{\sqrt{\delta_{i+1}}} - \frac{1}{\sqrt{\delta_{i}}} & = \frac{\sqrt{\delta_i} - \sqrt{\delta_{i+1}}}{\sqrt{\delta_{i+1}\delta_i}}  = \frac{\delta_i - \delta_{i+1}}{\delta_i\sqrt{\delta_{i+1}} + \delta_{i+1}\sqrt{\delta_i}}\\
        & \op{\delta_{i} \, \geq \, \delta_{i+1}}{\geq} \frac{\delta_i - \delta_{i+1}}{2\delta_i^{3/2}} \op{\cref{eq:assmp:boundedess_convexity}}{\geq} \frac{\delta_i - \delta_{i+1}}{2D^{3/2}\|\f{g}_i\|^{3/2}} \geq \frac{C}{2D^{3/2}}.
    \end{align*}
    Summing across all $i \in \mathcal{C}_{k}^{\complement}$, and using monotonicity of $\delta_i$, gives
    \[
        \frac{1}{\sqrt{\delta_k}} - \frac{1}{\sqrt{\delta_0}} > \frac{C}{2D^{3/2}}|\mathcal{C}_k^{\complement}| > \frac{Ck}{4D^{3/2}},
    \]
    which implies
    \begin{align}
        \label{eq:case_2}
        \delta_k < \left(\frac{CD^{-3/2}k}{4} + \sqrt{\delta_0}\right)^{-2}.
    \end{align}
    Combining \cref{eq:case_1,eq:case_2}, gives the desired result.  
\end{proof}

A special case of \cref{thm:reg_fncr} coincides with some recent results in the literature. Suppose we let $\omega = 0$ and $\theta = \sqrt{\|\f{g}_k\|}d$, which implies $\T = d$. From a basic property of Krylov subspace methods, the solution $\f{s}_k = -\f{H}_k^{-1}\f{g}_k$ will then be generated by \cref{alg:cr_with_sd_si}.  
From \cref{lem:descent_eta_region}, we can show that as long as $0 < \rho \leq 1/2 - \LH/(6\sigma^2)$, with $\sigma > \sqrt{\LH/3}$, the direction $\f{s}_k = -\f{H}_k^{-1}\f{g}_k$ is guaranteed to be $\rho$-sufficient. Indeed, suppose $\f{s}_k$ is not $\rho$-sufficient. By \cref{lem:descent_eta_region} with $h = \sigma\sqrt{\|\f{g}_k\|}$, if $\f{s}_k$ fails to be $\rho$-sufficient, then by \textit{modus tollens} we must have
\begin{align*}
    \sqrt{\frac{3(1 - 2\rho)}{\LH}}\,\frac{\sigma^{3/4}\|\f{g}_k\|^{3/8}}{|\lr{\f{g}_k,\f{s}_k}|^{1/4}} < 1,
\end{align*} 
which implies
\begin{align*}
    \left(\frac{3(1 - 2\rho)}{\LH}\right)^2\sigma^3\|\f{g}_k\|^{3/2} \leq \lr{\f{g}_k,\f{H}_k^{-1}\f{g}_k} 
    & < \frac{1}{\sigma\sqrt{\|\f{g}_k\|}}\|\f{g}_k\|^2.
\end{align*}
Hence $\left({3(1 - 2\rho)}/{\LH}\right)^2\sigma^4 < 1$, and therefore $\rho > 1/2 - \LH/(6\sigma^2)$, contradicting the premise. Thus, $\f{s}_k = -\f{H}_k^{-1}\f{g}_k$ must be $\rho$-sufficient, meaning one can ensure sufficient descent in the function value without performing any backtracking line-search operation. In particular, by setting $\sigma = \sqrt{\LH/2}$, the direction $\f{s}_k = -\f{H}_k^{-1}\f{g}_k$ is always $\rho$-sufficient for any $0 < \rho \leq 1/6$.  
Furthermore, note that
\begin{align*}
    \frac{2\|\f{g}_k\|}{\LH} = \frac{1}{(\sqrt{\LH\|\f{g}_k\|/2})^2}\|\f{g}_k\|^2 \leq \|\f{H}_k^{-1}\f{g}_k\|^2 = \|\f{s}_k\|^2.
\end{align*}
Therefore, to obtain the convergence rate for this specific case, we can appeal directly to \cref{lem:sufficient_inexactness}, setting $q_1 = 1$ and $q_2 = 0$, which yields the following result.

\begin{corollary}\label{cor:known_LH}
    Suppose $f$ is convex and \cref{assmp:LH,assmp:boundedness} hold. Let $0 < \rho \leq 1/6$, $\sigma = \sqrt{\LH/2}$, and $\TM = \T = d$. Then every iteration of FNCR-reg-LS (\cref{alg:fncr-ls} using $\f{H}_k$ from \cref{eq:H_reg}) satisfies
    \begin{align*}
        f_{k+1} < f_k - \frac{2\rho}{\sqrt{2\LH}} \|\f{g}_k\|^{3/2}.
    \end{align*}
    Furthermore, the overall worst-case iteration and operation complexities are $\mathcal{O}(1/\sqrt{\varepsilon})$ and $\mathcal{O}(d/\sqrt{\varepsilon})$.
\end{corollary}

Under the setting specified in \cref{cor:known_LH}, FNCR-reg-LS reduces to the method described in \cite[Algorithm~1]{mishchenko2023regularized}. In fact, \cite[Theorem~1]{mishchenko2023regularized} is a special case of \cref{cor:known_LH} with $\rho = \sqrt{2}/192$.

\section{Numerical Experiments}\label{sec:num_exp}
To evaluate the performance of our proposed methods, we consider the regularized cross-entropy loss function
\begin{align}\label{eq:cross_entropy_f}
    f(\f{x}) 
    &= \sum_{i = 1}^N \sum_{j = 1}^{C} -\delta(j,b_i)\log\left(\frac{\exp(\lr{\f{a}_i,\f{x}_j})}{\sum_{m=1}^C \exp(\lr{\f{a}_i,\f{x}_m)}}\right) + \mu\|\f{x}\|^2,
\end{align}
where $\{(\f{a}_i,b_i)\}_{i=1}^N \subset \mathbb{R}^d \times [C]$ denotes the input-output pairs, with $b_i$ being the class label from $C$ classes corresponding to input $\f{a}_i$; $\f{x} \triangleq [\f{x}_1^\top, \f{x}_2^\top, \dots, \f{x}_C^\top]^\top \in \mathbb{R}^{dC}$ is the parameter vector; $\delta(j,b_i) = 1$ if $b_i = j$ and $\delta(j,b_i) = 0$ otherwise; and $\mu \geq 0$ is the regularization parameter.  
Note that the problem is strongly convex when $\mu > 0$, and convex (but not strongly convex) when $\mu = 0$.

We will evaluate performance using four datasets: CIFAR10, CIFAR100~\cite{krizhevsky2009learning}, Covertype~\cite{covertype31}, and the Describable Textures Dataset (DTD)~\cite{cimpoi14describing}.  
The cross-entropy loss function is minimized using seven different methods: FNCR-LS (\cref{alg:fncr-ls} with \cref{alg:cr_with_sd_si} as subroutine), FNCR-reg-LS (\cref{alg:fncr-ls} with \cref{alg:cr_with_sd_si} subroutine, 
and the gradient-regularized Hessian, $\f{H}_k$, from \cref{eq:H_reg}), NewtonCG~\cite{nocedal2006numerical}, gradient-regularized Newton’s method with line-search~\cite{mishchenko2023regularized,doikov2024gradient,doikov2024super}, Steihaug’s trust-region method~\cite{steihaug1983conjugate}, L-BFGS~\cite{liu1989limited}, and gradient descent with line-search.  
These methods will be abbreviated as \texttt{FNCR-LS}, \texttt{FNCR-reg-LS}, \texttt{NewtonCG}, \texttt{GradReg}, \texttt{TR}, \texttt{L-BFGS}, and \texttt{GD}, respectively, in the plots. For each experiment, we present three plots: $f(\f{x})$ versus iteration, $f(\f{x})$ versus oracle calls,\footnote{To provide an implementation- and system-independent measure of complexity, we report results in terms of \emph{oracle calls}, where each function, gradient, or Hessian-vector product evaluation is expressed in units equivalent to a single function evaluation. Specifically, a gradient evaluation is counted as one additional function evaluation, and a Hessian-vector product as two additional evaluations~\cite{pearlmutterFastExactMultiplication1994}. This provides a fairer comparison across methods with varying per-iteration costs than wall-clock time, which can vary significantly with platform and implementation details.} and $f(\f{x})$ versus wall-clock time.  
The performance of these seven methods will be compared based on these plots. For both FNCR-reg-LS and FNCR-LS, we also track the number of \texttt{INS}-type directions used in each experiment, with the star symbol ($\star$) indicating the use of a \texttt{INS}-type direction at that iteration.

\begin{remark}
Before detailing the hyperparameter settings, we highlight a practical consideration regarding the implementation of \cref{alg:cr_with_sd_si} and its relation to \cref{thm:reg_fncr}. While theoretically clean, checking for sufficient descent and iteration conditions at every step, as in \cref{alg:cr_with_sd_si}, is inefficient in practice. For example, it is wasteful to verify that $\f{s}_0 = \f{0}$ is trivially sufficient, and inefficient to check $\rho_t$-sufficiency at every iteration $t \geq \T$.  
In our experiments with \texttt{FNCR-reg-LS} and \texttt{FNCR-LS}, we therefore perform sufficiency checks only every 20\textsuperscript{th} iteration after $\T$, i.e., at $t = \T + 20m$ for $m \geq 0$. For each successful check, the reduction $f(\f{x}_k) - f(\f{x}_k + \f{s}_t)$ is recorded, and the direction that yields the greatest reduction is stored. If insufficiency is detected at $m > 0$, a binary search is conducted within the interval $\T + 20(m-1) < t < \T + 20m$ to select the direction $\f{d}_k$ that provides the most reduction.
\end{remark}

The hyperparameters for each method are configured as follows:
\begin{itemize}
    \item \texttt{FNCR-reg-LS}: The regularization, sufficient descent, and inexactness parameters are set to $\sigma = 0.01$, $\rho = 0.01$, and $\omega = 0$, respectively. We set $\T = 5$ and $\TM = 1000$.
    
    \item \texttt{FNCR-LS}: Identical to \texttt{FNCR-reg-LS}, except with $\sigma = 0$.
    
    \item \texttt{GradReg}: Following the notation in \cite[Algorithm 2]{mishchenko2023regularized}, we set $H_0 = 0.01$.
    
    \item \texttt{L-BFGS}: The limited memory size is set to $20$, and the Strong Wolfe line-search parameter to $0.9$.
    
    \item \texttt{NewtonCG}: The inexactness parameter is set to $\omega = 0.1$, i.e., $\|\f{r}_t\| \leq \omega\|\f{g}_k\|$, where the Newton's system is solved using the CG method.
    
    \item \texttt{TR}: Using the notation and setup from \cite[Algorithm 4.1]{nocedal2006numerical}, we set $\Delta_0 = 10$, $\hat{\Delta} = 10^2$, and $\eta = 0.05$. The trust-region subproblem is approximately solved via the CG-Steihaug method \cite[Algorithm 7.2]{nocedal2006numerical}, with the inexactness parameter $\omega = 0.1$.
\end{itemize}

Unless stated otherwise, methods that use a backtracking line-search routine will have their parameters set to $\rho = 10^{-4}$, $\zeta = 0.5$, and initial step size $\eta_0 = 1$, except for GD where we set $\eta_0 = 0.01$. Initial weights $\f{x}_0$ are sampled from the uniform distribution $U[0, 1]$. Each method is terminated when the approximate first-order optimality condition $\|\f{g}_k\| \leq 10^{-6}$ is met, or the total number of oracle calls exceeds $10^5$.

We divide the four experiments into two regimes. The first regime includes the CIFAR10 (\cref{fig:cifar10}) and Covertype (\cref{fig:covtype}) datasets, where $N \gg dC$ and $\mu = 0.1$. Specifically, CIFAR10 has $N = 60{,}000$ and $dC = 30{,}720$, while Covertype has $N = 581{,}012$ and $dC = 378$. In this under-parameterized regime, \cref{eq:cross_entropy_f} is strongly convex.
The second regime includes the CIFAR100 (\cref{fig:cifar100}) and DTD (\cref{fig:dtd}) datasets, where $dC \gg N$ and $\mu = 0$. For CIFAR100, $N = 10{,}000$ and $dC = 307{,}200$, and for DTD, $N = 3{,}760$ and $dC = 3{,}172{,}500$. Here, \cref{eq:cross_entropy_f} corresponds to a convex, but not strongly convex, over-parameterized model.

\begin{figure}[htbp]
    \centering
    \includegraphics[width=1\linewidth]{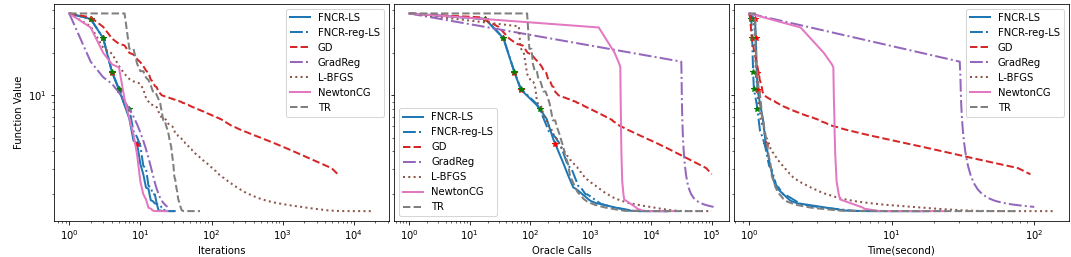}
    \caption{For CIFAR10, \texttt{FNCR-reg-LS}, \texttt{FNCR-LS}, \texttt{TR}, and \texttt{NewtonCG} terminate upon reaching the approximate first-order optimality condition. In contrast, \texttt{GradReg}, \texttt{L-BFGS}, and \texttt{GD} terminate after exceeding the maximum number of oracle calls. In this experiment, only 5 \texttt{INS}-type directions are used by both \texttt{FNCR-reg-LS} and \texttt{FNCR-LS}.}
    \label{fig:cifar10}
\end{figure}

\begin{figure}[htbp]
    \centering
    \includegraphics[width=1\linewidth]{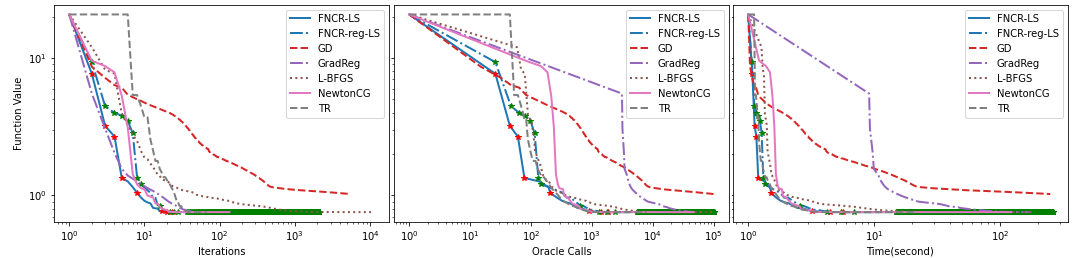}
    \caption{For Covertype, \texttt{FNCR-LS}, \texttt{GradReg}, and \texttt{NewtonCG} terminate upon achieving the approximate first-order optimality condition. In contrast, \texttt{FNCR-reg-LS}, \texttt{L-BFGS}, \texttt{TR}, and \texttt{GD} terminate after exceeding the maximum number of oracle calls. In this experiment, \texttt{FNCR-LS} uses 8 \texttt{INS}-type directions. However, when the gradient norm is approximately $1.01 \times 10^{-6}$, \texttt{FNCR-reg-LS} begins to use only \texttt{INS}-type directions with a very small step size, $\eta \approx 9.31 \times 10^{-10}$.}
    \label{fig:covtype}
\end{figure}

\begin{figure}[htbp]
    \centering
    \includegraphics[width=1\linewidth]{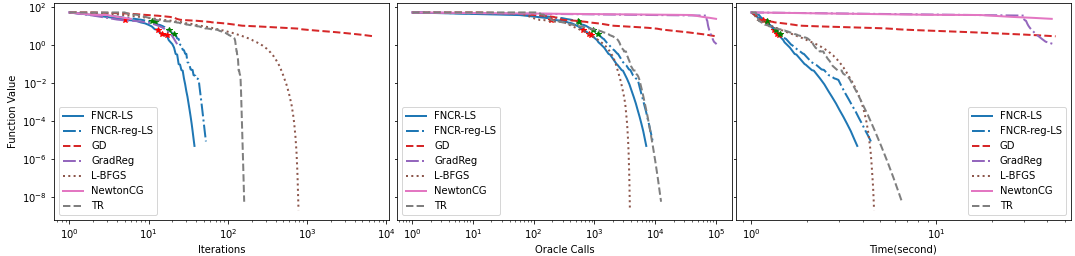}
    \caption{For CIFAR100, \texttt{FNCR-reg-LS}, \texttt{FNCR-LS}, \texttt{TR}, and \texttt{L-BFGS} terminate upon achieving the approximate first-order optimality condition. In contrast, \texttt{NewtonCG}, \texttt{GradReg}, and \texttt{GD} terminate after exceeding the maximum number of oracle calls. In this experiment, \texttt{FNCR-LS} uses three \texttt{INS}-type directions, while \texttt{FNCR-reg-LS} uses four.}
    \label{fig:cifar100}
\end{figure}

\begin{figure}[htbp]
    \centering
    \includegraphics[width=1\linewidth]{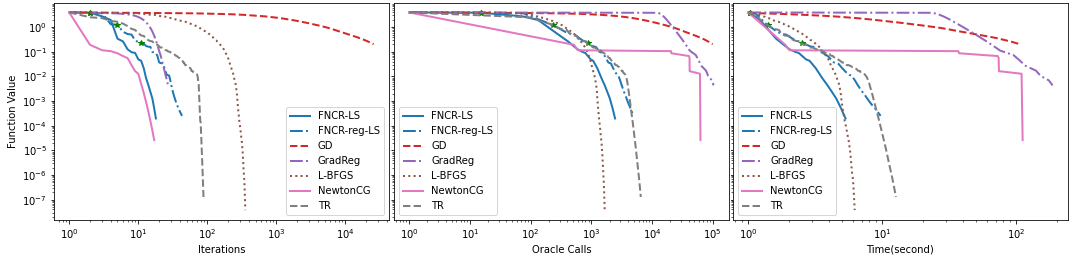}
    \caption{For DTD, \texttt{FNCR-reg-LS}, \texttt{FNCR-LS}, \texttt{NewtonCG}, \texttt{TR}, and \texttt{L-BFGS} terminate upon achieving the approximate first-order optimality condition. In contrast, \texttt{GradReg} and \texttt{GD} terminate after exceeding the maximum number of oracle calls. In this experiment, \texttt{FNCR-LS} uses one \texttt{INS}-type direction, while \texttt{FNCR-reg-LS} uses three.}
    \label{fig:dtd}
\end{figure}

As seen in \cref{fig:cifar10,fig:cifar100,fig:covtype,fig:dtd}, both \texttt{FNCR-LS} and \texttt{FNCR-reg-LS} perform competitively against several popular second-order methods in terms of oracle calls and runtime. To better understand the performance of \texttt{FNCR-LS} and \texttt{FNCR-reg-LS}, we compare them to \texttt{NewtonCG}. \texttt{NewtonCG} attempts to solve the linear system until the residual satisfies $\|\f{r}_k^{(t)}\| \leq 0.1\|\f{g}_k\|$, followed by a backtracking line-search. However, at iterates where the Hessian $\f{H}_k$ is ill-conditioned, \texttt{NewtonCG} may require many oracle calls to solve the linear system to the desired accuracy, resulting in high computational cost.
In contrast, both \texttt{FNCR-LS} and \texttt{FNCR-reg-LS} reduce this overhead by continuously monitoring the quality of the inexact Newton step $\f{s}_k^{(t)}$ based on its effect on the function values during the linear system solve. If $\f{s}_k^{(t)}$ fails to meet the $\rho_t$-sufficiency condition, \cref{alg:cr_with_sd_si} terminates early, allowing the outer solver to move on to the next iterate. This adaptive behavior is particularly efficient when $\f{H}_k$ is ill-conditioned or when the quadratic model \cref{eq:quadratic} is a poor local approximation at the current point. In such cases, as the Krylov subspace dimension increases, $\f{s}_k^{(t)}$ can quickly fail the $\rho_t$-sufficiency test, preventing unnecessary oracle calls and saving computation time.

The underperformance of \texttt{GradReg} may be explained similarly. The issue is more pronounced because, for every iteration of backtracking, \texttt{GradReg} needs to re-solve the linear system, making the algorithm extremely costly in practice, especially in high-dimensional problems.

Finally, across all experiments we observe that \texttt{FNCR-reg-LS} consistently performs slightly worse than \texttt{FNCR-LS}. This observation is consistent with the findings in \cite{lim2025complexity}, where it was shown that regularized methods tend to underperform compared with their non-regularized counterparts.

\subsection{Comparison with Accelerated Newton-Type Methods}\label{sec:exp:recent}
In this section, we compare our methods with several recently proposed advanced Newton-type variants that involve more complex subproblems, often nonlinear in nature. These include cubic regularized Newton with line-search (\texttt{CRN})~\cite{nesterov2006cubic}, accelerated cubic regularized Newton (\texttt{AccCRN})~\cite{nesterov2008accelerating}, accelerated cubic regularized Newton with adaptation (\texttt{NATA})~\cite[Algorithm~4]{kamzolov2024optami}, and optimal Monteiro--Svaiter acceleration with cubic regularization (\texttt{OptMS})~\cite[Algorithm~1 with Algorithm~3]{carmon2022optimal}. A common characteristic of these methods is the non-trivial nature of their subproblems, which are either nonlinear or require solving multiple linear systems, making them computationally demanding. To better illustrate the high cost of solving such subproblems, we evaluate these methods only on the DTD dataset, which is high-dimensional and thus particularly sensitive to expensive subproblem solves. Their performance is then compared with that of \texttt{FNCR-LS}, \texttt{FNCR-reg-LS}, and \texttt{GD}, where \texttt{GD} serves as a baseline.

Implementing \texttt{AccCRN} and \texttt{NATA} presents several technical challenges, particularly in selecting an appropriate upper bound $M$ such that $\LH \leq M$. Overestimating $M$ can slow convergence, while underestimating it may cause divergence. In our experiments, we approximate $\LH$ via a cold run of \texttt{CRN} with line-search, extracting the constant $M > 0$ that satisfies
\begin{align*}
    f(\f{x}_k + \f{d}_k) &\leq f(\f{x}_k) + \langle \f{g}_k, \f{d}_k \rangle + \frac{1}{2} \langle \f{d}_k, \f{H}_k \f{d}_k \rangle + \frac{M}{6} \|\f{d}_k\|^3,
\end{align*}
where $\f{H}_k = \He(\f{x}_k)$ and
\begin{align}\label{eq:cubic_subproblem}
    \f{d}_k \triangleq \arg\min_{\f{d} \in \Rd} f(\f{x}_k) + \lr{\f{g}_k,\f{d}} + \frac{1}{2}\lr{\f{d},\f{H}_k\f{d}} + \frac{M}{6}\|\f{d}\|^3.
\end{align}
This estimated $M$ is then used in \texttt{AccCRN} and \texttt{NATA}, while the \texttt{CRN} results in \cref{fig:dtd_acc} correspond to a warm run using this $M$.

Both \texttt{AccCRN} and \texttt{NATA} require solving two nonlinear subproblems (cubic and auxiliary). In high-dimensional settings, direct solvers are infeasible because only Hessian-vector products $\f{H}_k \f{d}$ are available. Following \cite{carmon2019gradient}, we solve these subproblems using gradient descent, up to a first-order suboptimality of $10^{-9}$.

The hyperparameters are set as follows:
\begin{itemize}
    \item \texttt{CRN}, \texttt{AccCRN}: $M = 1.5 \times 10^3$
    \item \texttt{NATA}: $\nu_{\min} = 10^{-6}$, $\nu_{\max} = 10$, $\theta = 2$, $M = 1.5 \times 10^3$
    \item \texttt{OptMS}: $\lambda_0' = 1.1$, $\sigma = 0.5$, \texttt{LAZY} = True
\end{itemize}

\begin{figure}[htbp]
    \centering
    \includegraphics[width=1\linewidth]{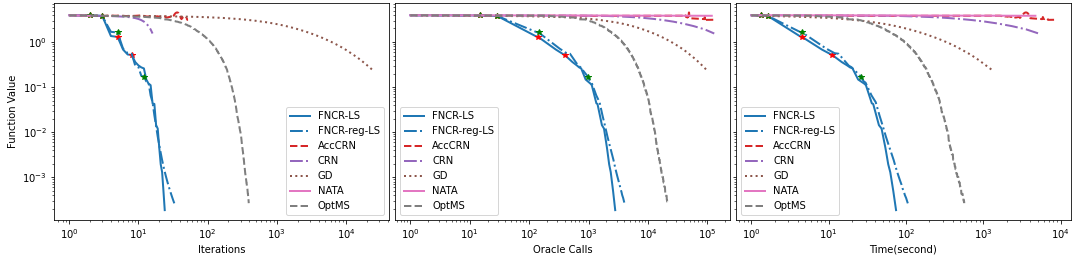}
    \caption{\texttt{FNCR-reg-LS}, \texttt{FNCR-LS}, and \texttt{OptMS} terminate upon satisfying the approximate first-order optimality condition. In contrast, \texttt{CRN}, \texttt{AccCRN}, and \texttt{NATA} terminate due to exceeding the maximum number of oracle calls. For both \texttt{FNCR-reg-LS} and \texttt{FNCR-LS}, four \texttt{INS}-type directions are used in this experiment.}
    \label{fig:dtd_acc}
\end{figure}


A major limitation of \texttt{CRN}, \texttt{AccCRN}, and \texttt{NATA} is their reliance on the nonlinear subproblem~\cref{eq:cubic_subproblem}, which is costly to solve. Each iteration of gradient descent applied to this subproblem requires at least two oracle calls, excluding additional evaluations from the backtracking line-search within gradient descent. For example, even after a preliminary run to estimate $\LH$, \texttt{CRN} typically required around $2{,}000$ iterations per subproblem solve. In the case of \texttt{NATA}, the algorithm reached its oracle call limit of $10^5$ during the first iteration while solving its subproblem, causing it to stop after just one iteration.

Moreover, both \texttt{AccCRN} and \texttt{NATA} include an additional nonlinear auxiliary subproblem. Although solving this auxiliary problem does not increase oracle calls--since it depends only on the gradient--it still adds substantial runtime, as seen in the runtime plot of \cref{fig:dtd_acc}. 

In contrast, both \texttt{FNCR-LS} and \texttt{FNCR-reg-LS} avoid solving such nonlinear subproblems, relying instead on adaptive monitoring of inexact Newton directions. This design makes them significantly more efficient in terms of oracle calls, runtime, and even iteration count, especially in high-dimensional settings where solving cubic or auxiliary subproblems becomes prohibitively expensive.

\section{Conclusion}
In this paper, we presented the \emph{Faithful-Newton} (FN) framework, which enables different instantiations of Newton-type methods by tightly integrating the inner and outer solvers---so that the choice of inner solver directly determines the convergence guarantees of the outer iteration. As a concrete instance, we proposed the Faithful-Newton-Conjugate-Residual with Line-search (\texttt{FNCR-LS}) method. We demonstrated that, under strongly convex and Lipschitz-Hessian conditions, \texttt{FNCR-LS} achieves either global superlinear or condition-number-independent linear convergence, depending on parameter choices. Locally, it attains quadratic convergence even with inexact Newton steps. Moreover, we showed that classical inexact and truncated Newton's methods arise as special cases of \texttt{FNCR-LS}, and that our superlinear convergence results extend naturally to these variants---thereby challenging the longstanding perception that Newton’s method exhibits worse worst-case behavior than gradient methods. We further extended our analysis to general convex functions through a regularized variant, \texttt{FNCR-reg-LS}, and established an iteration complexity matching the best-known rates among non-accelerated methods. Finally, we validated our theoretical results through comprehensive empirical studies, demonstrating the strong practical performance of \texttt{FNCR-LS} and \texttt{FNCR-reg-LS} compared with several alternative second-order methods.

\bibliographystyle{plain}
\bibliography{ref}

\appendix
\section{Deferred Proofs}\label{sec:deferred_proofs}
\begin{proof}[Proof of \cref{lem:descent_eta_region}]
    By \cref{eq:assmp:LH:fkp1<fk}, we have
    \begin{align*}
        f(\f{x}_k + \eta\f{s}_k) & \leq f_k + \eta\lr{\f{g}_k,\f{s}_k} + \frac{\eta^2}{2}\lr{\f{s}_k,\He(\f{x}_k)\f{s}_k} + \frac{\LH\eta^3}{6}\|\f{s}_k\|^3,
    \end{align*}
    where $\f{s}_k$ is a vector generated by \cref{alg:cr_with_sd_si} with $\f{H}_k$ and $\f{g}_k$. Consider first $\f{H}_k = \He(\f{x}_k)$ and, so, $\lr{\f{d},\f{H}_k\f{d}} \geq h\|\f{d}\|^2$. This implies,
    \begin{align*}
         f(\f{x}_k + \eta\f{s}_k) - f_k - \rho\eta\lr{\f{g}_k,\f{s}_k} & \leq (1 - \rho)\eta\lr{\f{g}_k,\f{s}_k} + \frac{\eta^2}{2}\lr{\f{s}_k,\f{H}_k\f{s}_k} + \frac{\LH\eta^3}{6}\|\f{s}_k\|^3\\
        & \hspace{-16mm} \op{\cref{eq:gst<-sHs}}{\leq} -(1 - \rho)\eta|\lr{\f{g}_k,\f{s}_k}| + \frac{\eta^2}{2}|\lr{\f{g}_k,\f{s}_k}| + \frac{\LH\eta^3h^{3/2}}{6}|\lr{\f{g}_k,\f{s}_k}|^{3/2}\\
        & \hspace{-16mm}\op{(\eta \, \leq \, 1)}{\leq} \eta|\lr{\f{g}_k,\f{s}_k}|\left(-(1 - \rho) + \frac{1}{2} + \frac{\LH\eta^2h^{3/2}}{6}|\lr{\f{g}_k,\f{s}_k}|^{1/2}\right).
    \end{align*}
    When the step-size $\eta$ satisfies \cref{eq:lem:descent_eta_region}, the right-hand side of the above inequality becomes non-positive, which implies $\eta\f{s}_k$ is $\rho$-sufficient. 
    
    Now consider the regularized Hessian $\f{H}_k = \He(\f{x}_k) + h\f{I}$, for $h > 0$. We have 
    \begin{align*}
        f(\f{x}_k + \eta\f{s}_k) - f_k - \rho\eta\lr{\f{g}_k,\f{s}_k} & \leq (1 - \rho)\eta\lr{\f{g}_k,\f{s}_k} + \frac{\eta^2}{2}\lr{\f{s}_k,\He(\f{x}_k)\f{s}_k} + \frac{\LH\eta^3}{6}\|\f{s}_k\|^3\\
        & \hspace{-16mm} = (1 - \rho)\eta\lr{\f{g}_k,\f{s}_k} + \frac{\eta^2}{2}\lr{\f{s}_k,\f{H}_k\f{s}_k} - \frac{\eta^2 h}{2}\|\f{s}_k\|^2 + \frac{\LH\eta^3}{6}\|\f{s}_k\|^3\\
        & \hspace{-16mm} \leq (1 - \rho)\eta\lr{\f{g}_k,\f{s}_k} + \frac{\eta^2}{2}\lr{\f{s}_k,\f{H}_k\f{s}_k} + \frac{\LH\eta^3}{6}\|\f{s}_k\|^3.
    \end{align*}
    By the same reasoning as before, this again implies that if the step-size $\eta$ satisfies \cref{eq:lem:descent_eta_region}, $\eta\f{s}_k$ is $\rho$-sufficient.
\end{proof}

\begin{proof}[Proof of \cref{lem:sufficient_inexactness}]
    Suppose $\|\f{s}_k\|^2 > 2q_1\|\f{g}_k\|/\LH$. Since $\f{s}_k$ is $\rho$-sufficient, we have
    \begin{align*}
        f(\f{x}_k + \f{s}_k) - f_k \leq \rho\lr{\f{g}_k,\f{s}_k} \op{\cref{eq:gst<-sHs}}{\leq} - \rho\sigma\sqrt{\|\f{g}_k\|}\|\f{s}_k\|^2 < - \frac{2\rho\sigma q_1}{\LH}\|\f{g}_k\|^{3/2}.
    \end{align*}
    Otherwise, when $2q_1\|\f{g}_k\|/\LH \geq \|\f{s}_k\|^2$, given $q_2\|\f{g}_k\| \geq \|\f{r}_k\|$, we have
    \begin{align*}
        \|\f{g}(\f{x}_k + \f{s}_k)\| & = \|\f{g}(\f{x}_k + \f{s}_k) - \f{g}_k - \f{H}_k\f{s}_k - \f{r}_k\| \\
        &= \|\f{g}(\f{x}_k + \f{s}_k) - \f{g}_k - \He(\f{x}_k)\f{s}_k - \sigma\sqrt{\|\f{g}_k\|}\f{s}_k - \f{r}_k\|\\
        & \leq \|\f{g}(\f{x}_k + \f{s}^{(t)}) - \f{g}_k - \He(\f{x}_k)\f{s}_k\| + \sigma\sqrt{\|\f{g}_k\|}\|\f{s}_k\| + \|\f{r}_k\|\\
        & \op{\cref{eq:assmp:LH:gkp1<p}}{\leq} \frac{\LH}{2}\|\f{s}_k\|^2 + \sigma\sqrt{\|\f{g}_k\|}\|\f{s}_k\| + \|\f{r}_k\| \\
        &\leq \frac{\LH}{2}\|\f{s}_k\|^2 + \sigma\sqrt{\|\f{g}_k\|}\|\f{s}_k\| +  q_2\|\f{g}_k\|\\
        & \leq q_1\|\f{g}_k\| + \sigma \sqrt{\frac{2q_1}{\LH}}\|\f{g}_k\| + q_2\|\f{g}_k\| \leq \left(q_1 + \sigma \sqrt{\frac{2q_1}{\LH}} + q_2\right)\|\f{g}_k\|.
    \end{align*}
    The inequality $f(\f{x}_k + \f{s}_k) < f_k$ follows from the assumption that $\f{s}_k$ is $\rho$-sufficient.
\end{proof}

\begin{proof}[Proof of \cref{lem:obj_descents}]
    Since $\f{s}_k$ is not $\rho$-sufficient, by \textit{modus tollens}, \cref{lem:descent_eta_region} with $h = \sigma\sqrt{\|\f{g}_k\|}$ implies
    \begin{align*}
        \left(\frac{3(1 - 2\rho)}{\LH}\right)^2\sigma^3\|\f{g}_k\|^{3/2} < \left|\lr{\f{g}_k,\f{s}_k}\right|.
    \end{align*}
    Hence, \cref{alg:backtracking} returns $\eta$ such that 
    \begin{align*}
        f(\f{x}_k + \eta\f{s}_k) - f_k & \leq \rho\eta\lr{\f{g}_k,\f{s}_k} \op{\cref{eq:backtracking_eta}}{\leq} - \rho\zeta\sqrt{\frac{3(1 - 2\rho)}{\LH}}\sigma^{3/4}\|\f{g}_k\|^{3/8}\left|\lr{\f{g}_k,\f{s}_k}\right|^{3/4}\\
        & < - \rho\zeta\left(\frac{3(1 - 2\rho)}{\LH}\right)^2\sigma^{3}\|\f{g}_k\|^{3/2}.
    \end{align*}
\end{proof}

\end{document}